\documentclass[onefignum,onetabnum,referee]{siamart171218}
\usepackage{multicol}
\newcommand{\T}{T}
\newcommand{\R}{\mathbb{R}}
\newcommand{\Rex}{\overline{\mathbb{R}}}
\newcommand{\norma}{\|\cdot\|}
\usepackage{amssymb}
\usepackage{cite}
\usepackage{enumitem}
\newcommand{\tto}{\rightrightarrows}
\newtheorem{example}[theorem]{Example}
\let\varepsilon\varepsilon 
\DeclareMathOperator{\Ls}{Ls}
\newcommand{\N}{\mathbb{N}}

\DeclareMathOperator{\co}{co}
\DeclareMathOperator{\cco}{\overline{co}}

\DeclareMathOperator{\grafo}{gph}
\DeclareMathOperator{\dom}{dom}

\DeclareMathOperator{\inte}{int}
\DeclareMathOperator{\spn}{span}

\newcommand{\Intf}[1]{{E}_{#1}}
\newcommand{\IntfLp}[1]{{I}_{#1}}
\DeclareMathOperator{\Asub}{\hat{\partial}}
\DeclareMathOperator{\sub}{\partial}
\usepackage{dsfont}
\newcommand{\1}{\mathds{1}}
\let\epsilon\varepsilon

\usepackage{lipsum}
\usepackage{amsfonts}
\usepackage{graphicx}
\usepackage{epstopdf}
\usepackage{algorithmic}
\ifpdf
  \DeclareGraphicsExtensions{.eps,.pdf,.png,.jpg}
\else
  \DeclareGraphicsExtensions{.eps}
\fi

\newsiamremark{remark}{Remark}
\newsiamremark{hypothesis}{Hypothesis}
\crefname{hypothesis}{Hypothesis}{Hypotheses}
\newsiamthm{claim}{Claim}

\headers{Subdifferential of nonconvex integrals}{R. Correa, A. Hantoute and P. P\'erez-Aros}

\title{Sequential and exact subdifferential calculus rules for nonconvex integral functions\thanks{Submitted to the editors DATE.\funding{This work is partially supported by CONICYT grants: Fondecyt 1151003, Fondecyt 1150909, Basal PFB -03 and Basal FB0003, CONICYT-PCHA/doctorado Nacional / 2014-21140621.}}}

\author{Rafael Correa\thanks{Universidad de O'Higgins, Rancagua,  Chile and DIM-CMM of Universidad de Chile, Santiago, Chile 
		(\email{rcorrea@dim.uchile.cl}).}
	\and Abderrahim Hantoute\thanks{Center for Mathematical Modeling (CMM), Universidad de Chile, Santiago, Chile
		(\email{ahantoute@dim.uchile.cl}).}
	\and Pedro P\'erez-Aros\thanks{Instituto de Ciencias de la Ingenier\'ia, Universidad de O'Higgins, Rancagua, Chile 
		(\email{pedro.perez@uoh.cl}).} }

\usepackage{amsopn}

\ifpdf
\hypersetup{
  pdftitle={},
  pdfauthor={R. Correa, A. Hantoute, and P. P\'erez-Aros}
}
\fi

\begin{document}

\maketitle

\begin{abstract}
 We are concerned in this work  with  the subdifferential of the integral functional 
 $$
 \Intf{f}(x)=\int_{\T} f(t,x)d\mu(t),
 $$ 
 for normal integrands $f$ (possibly, not convex), defined in a $\sigma$-finite measure space and a separable Banach space whose dual is also separable. By means of techniques of variational analysis, we establish some sequential formulae to estimate  the Fr\'echet  subdifferential  of $\Intf{f}$. We also give upper-estimates for the limiting subdifferential of $\Intf{f}$ for both Lipschitz and non-Lipschitz integrands. This last result is based on a Lipschitz-like condition.
 \end{abstract}


\section{Introduction\label{Sect1}}
 We are interested in this paper to the variations at a first-order of the integral function given in the form 
$$
\Intf{f}(x)=\int_{\T} f(t,x)d\mu(t),
$$ 
for a normal integrand $f:{\T}\times X\to \mathbb{R}\cup\{+\infty\}$, which is defined on a measurable space $({\T},\mathcal{A},\mu)$ and an infinite-dimensional Banach space $(X,\| \cdot \|)$. The normal integrand is possibly non-convex. Our aim  is to provide sequential and exact formulae of the nonconvex  subdifferentials of the integral function $\Intf{f}$, including the Fr\'echet, the limiting, and the Clarke-Rockafellar subdifferentials. This will be achieved by means of estimates that make use of the data, namely the measurable selections of the corresponding  subdifferential of the integrand $f$. The function $\Intf{f}$ can also be regarded as a functional operator acting on the subspace of constant functions on $T$, as in \cite{0036-0279-23-6-R02} and \cite{MR0390870} (see, also, \cite{MR0236689,MR0310612,MR0372610,0036-0279-30-2-R03,MR0467310,MR2291564,MR2444461,MR3235316,MR3767752}), but here in the current work $\Intf{f}$ is considered as such; that is, as a continuous sum. 

 This work is inspired by the results of Ioffe \cite{MR2291564} and  Lopez-Thibault \cite{MR2444461}. In \cite{MR2291564}, where the author deals with convex normal integrands,  one can find characterizations of the Fenchel subdifferential of the integral function $\Intf{f}$, given by means of limiting processes relying on the subdifferential of the data. Such characterizations do not involve any qualification condition. Namely, assuming that $f$ is a convex normal integrand, satisfying a mere linear growth condition, given as
\begin{align}
	f(t,x) \geq \langle a^*(t) , x \rangle + \alpha(t) ,\; \text{ for all } t\in {\T}, \; x\in X,
\end{align}
for  integrable functions $a^* : {\T}\to X^\ast$  and  $\alpha :{\T} \to \R$, the subdifferential of  $\Intf{f}$ at a given point $x\in X$ can be characterized  in terms of  limits of the integral of measurable selections of the Fenchel subdifferential of the integrand $f$ at nearby points $x_n(t)$, converging to $x$ in an appropriate way. More precisely, it is proved in \cite{MR2291564} that the sequence  of the measurable functions $x_n(\cdot)$ can be taken in the space of $p$-integrable functions for any $p \in [1,+\infty)$. The question of whether the same property holds for the case when $p=+\infty$ is treated in Lopez-Thibault \cite{MR2444461}, where the authors provided another approach to this problem, by using convex fuzzy-calculus subdifferential rules. 

We continue this line of research by deriving new sequential and exact formulae for the subdifferential of the integral function $\Intf{f}$ for non-convex normal integrands $f$. Our approach consists in using the concept of \emph{robust infima} (see \cref{ROBUSTEDLOCALMINIMUM}), which we combine with some variational principles, applied in the space $X$ as well as in the functional space of $p$-integrable functions. Using this we  extend and improve the results of   \cite{MR2291564} and \cite{MR2444461} (see also \cite{INTECONV,INTECONV2}). 

Let us mention that all the results given in this paper have been developed in \cite{modelacionsubdifferential}. Related results can be found in \cite{MR3767752,MR2772136,MR3639281}.  Here, for the sake of simplifying the presentation, we only give the aforementioned results for the Fr\'echet and the limiting subdifferentials, both based on the notion of robust infima, instead of the common approach using chain rules as developped in  \cite{MR2444461}.


This work is organized as follows: \Cref{NotationCHAP5} is dedicated to recall some notions of variational analysis and the generalized subdifferentiation that are needed in the sequel. In \Cref{section:Robusted} we adapt to our setting the notion of robust local minima (see  \cref{ROBUSTEDLOCALMINIMUM}), which allows applying Borwein-Preiss' variational principle. Next,  we give sequential formulae for  the Fr\'echet subdifferential of the integral $\Intf{f}$ (see  \cref{theoremsubdiferential,theoremsubdiferentialinfinity,corollary:convex}). In  \Cref{Section:Limitingsub} we introduce a Lipschitz-like condition, which  generalizes the classical Lipschitz continuity of integral functionals (see, e.g., \cite[Theorem 2.7.2]{MR1058436}), and  leads to upper-estimates for the limiting subdifferential, as well as for the Clarke-Rockafellar subdifferential of the integral functional. Finally, for the sake of simplifying the presentation of the work, the technical results and proofs are given in the Appendix.

\section{Notation and preliminary results}\label{NotationCHAP5}

In the following, $(X,\| \cdot \|)$ will be a separable Asplund space and $X^*$ its dual, which means that $X^\ast$ is also separable (see, e.g., \cite{MR2766381} for more details). The norm in $X^\ast$ will be denoted also by $\| \cdot \|$. For a point $x\in X$ and $r>0$, the closed ball of radius $r$ and centered  at $x$ is denoted by $\mathbb{B}_X(x,r)$, or simply $\mathbb{B}(x,r)$ when no confusion occurs, particularly, the unit closed ball is simply denoted by $\mathbb{B}$. The bilinear form $ \langle \cdot ,\cdot \rangle :X^*\times  X \to \R$ is given by $\langle x^*,x \rangle :=\langle x,x^*\rangle :=x^*(x) $. The weak$^{\ast }$-topology on $X^*$ is  denoted by $w(X^{\ast },X)$ ($w^{\ast },$ for
short).  We write $\Rex:= \R\cup\{-\infty,+\infty \}$ with  the conventions  $1/\infty=0$, $0\cdot \infty= 0 = 0\cdot(- \infty)$ and $ \infty+( -\infty)=(-\infty)+\infty=+\infty$. 

For a set $A\subseteq X$ (or $\subseteq X^*$), we denote by $\inte(A)$, $\overline{A}$, $\co(A)$, and $\cco(A)$  the interior, the closure, the \emph{convex hull} and  the \emph{closed convex hull} of $A$, respectively. The linear space spanned by $A$ is denoted by   $\spn(A)$ and the \emph{negative polar cone} of $A$ is the set 
\begin{align*}
A^- &:= \{ x^*\in X^\ast \mid \langle x^*,x\rangle \leq 0, \forall x\in A\}.
\end{align*}
The \emph{indicator} and the \emph{support} functions of a set $A$ ($\subseteq X,X^\ast$) are, respectively,
\[  \delta_A(x) := \begin{cases}
0\qquad & x\in A,\\
+\infty& x\notin A,
\end{cases} \text{ and } \sigma_A(x^*):=\sup\{  \langle x^*,x \rangle : x\in A \}. \]

For a given function $f:X\to\Rex$, the (effective) \emph{domain} of $f$ is $\dom f := \{ x\in X\mid f(x) < +\infty \}$. We say that $f$ is \emph{proper} if $\dom f \neq\emptyset$ and $f>-\infty$, and  \emph{sequentially $\tau$-inf-compact} (for some topology $\tau$ on $X$) if for every $\lambda \in \R$ and every sequence $(x_n)  \subset [f\le \lambda]:=\{ x \in X\mid f(x)\leq \lambda \}$ there exists a subsequence  $x_{n_k} \overset{\tau}{\to} x\in[f\le \lambda]$.  We write $ x \overset{f}{\to} x_0\in X$ to say that $x\to x_0$ with $f(x) \to  f(x_0)$.


Now, we consider a function $f:X\to \Rex$, which is finite at $x$. Then  the \emph{Fr\'echet} (or \emph{regular}) subdifferential  of $f$ at $x$ is defined by  
\begin{align*}
\hat{\partial} f(x):=& \Bigg\{ x^*\in X^\ast \mid \liminf\limits_{h \to 0} \frac{f(x+h) -f(x) - \langle x^*, h\rangle}{\| h\|} \geq 0  \Bigg\}.
\end{align*}
Since the space $X$ is  Fr\'echet smooth, that is, it has an equivalent norm that is $C^1$  for all $x \neq 0$ (see, e.g., \cite[Theorem 8.6]{MR2766381}),  we have that the  \emph{Fr\'echet} subdifferential coincide with the \emph{Viscosity Fr\'echet} Subdifferential,  that is to say,  $x^\ast \in \hat{\partial} f(x)$ if and only if  there are neighbourhood $U$ of $x$ and a $C^1$ function  $\phi:U \to \R$  such that $ \nabla \phi(x) = x^*$ and $f -  \phi$ attains its local minimum at $x$ (see, e.g., \cite{MR2144010,MR2986672}).

The \emph{limiting} (or \emph{basic}, or \emph{Mordukhovich} subdifferential)  and the \emph{singular} subdifferentials  are defined as (see, e.g., \cite{MR2191745,mordukhovich2018variational,MR2191744})
\begin{align*}
\partial f(x) &: =\Bigg\{ w^*\text{-}\lim x^*_n :
x_n^* \in \Asub f(x_n),  \text{ and } x_n \overset{f}{\to}x \Bigg\}
,\\ 
\partial^\infty f(x) &:=  \Bigg\{ w^*\text{-}\lim \lambda_n x^*_n : 
 x_n^* \in \Asub  f(x_n), \; x_n \overset{f}{\to}x \text{ and } \lambda_n \to 0^+
   \Bigg\},
\end{align*}
respectively. Finally, the \emph{Clarke-Rockafellar} subdifferential can be defined as (see, e.g., \cite[Theorem 3.57]{MR2191744}) 
\begin{align}\label{definition:clarke}
	\sub_C f(x):=\cco\left\{  	\partial f(x) + \partial^\infty f(x) \right\}.
	\end{align}

 If $|f(x)| =+\infty$, we set $\Asub f(x):=\partial f(x):=\sub^\infty f(x):=\sub_C f(x):=\emptyset$.

It is important to emphasize that when  $f$ is lower-semicontinuous (lsc),  proper  and  convex,   the Fr\'echet, the limiting and the Clarke-Rockafellar   subdifferentials coincide with the convex (or \emph{Fenchel}, \emph{Moreau-Rockafellar}) subdifferential given for $x\in  \dom f$ by 
$$
\sub f(x):=\{ x^* \in X^* : \langle x^* , y -x \rangle \leq f(y) -f(x), \forall y \in X  \}.
$$

Throughout the paper,  $({\T},\mathcal{A},\mu)$ is a complete $\sigma$-finite measure space.  A function $x : {\T} \to X\ (\text{or } X^*)$ is called \emph{simple} if there are  $k \in \N$, sets  $T_i \in \mathcal{A}$  and elements $x_i \in X$, $i=0,...,k$, such that $x(\cdot) = \sum_{i=0}^k x_i \1_{T_i}(\cdot)$, where $\1_{T_i}(t)=1$ if $t\in T_i$ and $\1_{T_i}(t)=0$ if $t\notin T_i$. A function $x: {\T}\to X(\text{or } X^*)$ is called \emph{measurable} if there exists a  countable family $x_n$ of simple functions such that $\lim_{n \to \infty} \| x (t) - x_n(t)\|=0$ $\mu$-almost everywhere  (ae, for short).     For $p \in [1,\infty)$ we denote by $\textnormal{L}^p({\T},X)$ and $\textnormal{L}^p({\T},X^*)$  the sets of all  (equivalence classes by the relation $f=g$ ae)  measurable  functions $f$ such that $\| f(\cdot)\|^p$ is integrable. As usual, the corresponding norm in these spaces is $\| f\|_p :=(\int_{\T} \|f(t)\|^p d\mu(t))^{1/p}$. For  an integrable function $x^*(\cdot)$ and a measurable set $A$, the symbols $\int_A x^*(t)d\mu(t)$ denote the  \emph{Bochner} integral of  $x^*$ over $A$  (see \cite[\S II. Integration]{MR0453964} and the  details therein).  The space $\textnormal{L}^\infty({\T},X)$ consists of all (equivalent classes with respect  to the the relation $f=g$ ae) measurable and essentially bounded functions  $x: {\T}\to X$. The associated norm on $\textnormal{L}^\infty({\T},X)$ is given by $\| x \|_\infty := \text{ess}\sup_{t\in {\T}} \| x(t)\|$. In a similar way, the space $\textnormal{L}^\infty({\T},X^*)$ is defined by the set of all measurable essentially bounded function  $x^*: {\T}\to X^*$. 

A sequence of functions $(\varphi_k)_{k\in\N} \subseteq  \textnormal{L}^1({\T},\R)$ is said to be uniformly integrable if $$\lim\limits_{ a\to \infty}\sup\limits_{k} \displaystyle\int\limits_{ \{| \varphi_k(t)| \geq a \}} | \varphi_k(t)|    d\mu(t)=0.$$

A function  $f:{\T}\times X \to \Rex$  will be called  a \emph{normal integrand} if it is  $\mathcal{A}\otimes\mathcal{B}(X)$-measurable (where $\mathcal{B}(X)$ is the  \emph{Borel} $\sigma$-algebra, i.e., the $\sigma$-algebra generated by all open sets of $X$) and  for every $t\in {\T}$,  $f_t:=f(t,\cdot)$ is lsc. For $\sub$ being any one of the subdifferentials above and any measurable function $x :{\T} \to X$, we denote $\sub f(t,x(t)) :=\sub f_t(x(t))$; i.e., the subdifferentials are taken with respect to the second variable.

  An  \emph{integral functional} on $\textnormal{L}^p(T,X)$ (with  $p\in [1,+\infty]$) is an extended-real-valued functional $\IntfLp{f}^{\mu,p}:\textnormal{L}^p(T,X) \to \Rex$  of the form 
\[ \displaystyle x\to \IntfLp{f}^{\mu,p}(x(\cdot)):=\int\limits_T f(t,x(t)) d\mu(t):=\int\limits_T f(t,x(t))^{+}d\mu(t)+ \int\limits_Tf(t,x(t))^{-} d\mu(t),\] 
where   $ \alpha^{+} := \max\{ \alpha,0\} $ and $\alpha^{-}:=\min \{ \alpha,0\}$; we simply write  $\IntfLp{f}$, when there is no ambiguity.   

The main concern of this paper  is the study of the subdifferential of the following particular class of integral functionals (also called  continuous sum or integral sum), $\Intf{f}: X\to \Rex$, given by
$$
\Intf{f}^\mu(x):= \int_{\T} f(t,x) d\mu(t);
$$ 
we simply write  $\Intf{f}$, when there is no confusion. It is worth mentioning that $\Intf{f}$ can be understood as  the  integral functional  $\IntfLp{f}$ restricted to the constant measurable functions.

The subdifferential theory of functions $\IntfLp{f}$ and $\Intf{f}$ goes back to Ioffe-Tikhomirov		 \cite{0036-0279-23-6-R02} and  Rockafellar \cite{MR0390870} for convex normal integrands. Posteriorly, these functionals have been considered by several authors; for example, Rockafellar \cite{MR0236689,MR0310612}, Ioffe-Levin \cite{MR0372610}, Levin \cite{0036-0279-30-2-R03}, Castaing-Valadier \cite{MR0467310},  Ioffe \cite{MR2291564}, Lopez-Thibault \cite{MR2444461},  Borwein-Yao \cite{MR3235316}, and Ginner-Penot \cite{MR3767752}, among others.

Now, let us recall the concept of a \emph{graph measurable}  multifunction. We recall that a Hausdorff topological space $S$ is  a Suslin space  provided that there
exists a Polish space $P$ (complete, metrizable and separable) and a
continuous surjection from $P$ to $S$ (see \cite%
{MR0027138,MR0467310,MR0426084}). Typical examples of suslin spaces are a sepable Banach space (with the norm topology) and its dual (with the weak$^*$-topology).

Consider a Suslin space $S$. A multifunction $M: {\T} \tto S$ is said to be \emph{graph measurable} (or simply measurable) if its graph, $\grafo M :=\{(t,s)\in {\T}\times S : s\in M(t)  \}$, belongs to $ \mathcal{A} \otimes \mathcal{B}(S)$ (see, e.g.,  \cite{MR2458436,MR0467310,MR1485775,MR0499053,MR0486391,MR577971} for more details). 

The next proposition corresponds to the  \emph{Measurable Selection Theorem} for graph measurable multifunction with values in Suslin spaces.
\begin{proposition}\cite[Theorem III.22]{MR0467310}\label{measurableselectiontheorem}
	Consider a graph measurable multifunction $M: T\tto S$ with non-empty values. Then the exists a measurable function $m: T \to S$ such that $m(t) \in M(t)$ for all $t\in T$.
	\end{proposition}
 
 For a multifunction $M:{\T}\tto X^*$ and a measurable set $A \in \mathcal{A}$, we define \emph{the Bochner integral of $M$ over $A$} by   $$\int\limits_A M(t) d\mu(t) :=\Bigg\{  \int\limits_A x^*(t) d\mu(t) : x^* \in \textnormal{L}^1({\T},X^*) \textnormal{ and } x^*(t) \in M(t) \text{ ae}       \Bigg\}. $$
It is worth recalling that the original definition of integral of set-valued mappings is due to R. J. Aumann, given for multifunctions defined on closed intervals in $\R$ (see, for example, \cite{MR0185073}). 

Unless stated otherwise,  in the rest of this article we  assume that $f$ is an integrand from ${\T}\times X$  to $ [0,+\infty]$,  $X$ is  a separable Asplund space and its  norm is $C^1$ away from the origin, we refer to  \cite{MR2766381} for more details about the theory of Asplund spaces. Although the assumption about the range of the values of the integrand appears less general, many of the results in the literature can be obtained in our setting by modifying appropriately the integrand. We will talk more in depth about these techniques in    \cref{theoremsubdiferential,theoremsubdiferentialinfinity}. It is important to recall that in our  framework the integral functional $\IntfLp{f} : (\textnormal{L}^p({\T},X),\| \cdot \|_p) \to \Rex$ is lsc (see, e.g., \cite[Lemma 10]{MR3767752}).

The next proposition corresponds to an extension of the well-known result of Rockafellar concerning the interchange between the infimum and  integral in finite-dimensional spaces using the concept of a decomposable space (see, e.g., \cite{MR0236689,MR0310612,MR0390870,MR1491362}). Lately, this  theorem  was extended to separable infinite-dimensional spaces \cite[Theorem VII-7]{MR0467310}. We also refer to  \cite{INTECONV,MR3767752,MR2559616,MR0507504} for other versions of this result. 
\begin{proposition}\label{dualitycastaing} Let  $h:{\T}\times X\rightarrow \overline{\mathbb{R}}$ be a
	normal integrand such that $\IntfLp{h}(u_{0})<\infty $ for some $u_{0}\in
	\textnormal{L}^p({\T},X)$. Then 
	\begin{equation*}
	\inf\limits_{v\in \textnormal{L}^p({\T},X)} \int\limits_T h(t,v(t)) d\mu(t) =\int\limits_T \inf\limits_{x\in X} h(t,x)d\mu(t).
	\end{equation*}
\end{proposition}

\section{Sequential formulae for the subdifferential  of integral functions\label{formulae}}\label{section:Robusted}

In order to deal with an arbitrary complete $\sigma$-finite measure space $({\T},\mathcal{A},\mu)$, we adapt here the notion of \emph{robust local minima} or \emph{decoupled infima} (see, for example, \cite{MR2191744,MR2144010,MR2986672,MR3033176}) to the case of  integral functionals. 
\begin{definition}\label{ROBUSTEDLOCALMINIMUM}
Let $({\T},\mathcal{A},\mu)$ be a finite measure space.	Consider a function $f:{\T}\times X \to \Rex$ and $p\in [1,+\infty)$. We define the  $p$-stabilized infimum of  $\Intf{f}$ on $B \subseteq X$  by
	$$\wedge_{p,B} \Intf{f} :=\sup\limits_{\varepsilon > 0}\inf\Bigg\{ \displaystyle \int\limits_T f(t,x(t)) d\mu(t) \mid \begin{array}{c}
	x(\cdot) \in \textnormal{L}^p({\T},X), \; \; y \in B\\
	 \text{ and } \displaystyle\int\limits_T \| x(t) -y\|^p d\mu(t)  \leq  \varepsilon 
	\end{array}    \Bigg\}.$$
	The infimum of $\Intf{f}$ on $B$ is called $p$-robust if $\wedge_{p,B} \Intf{f} =\inf_B \Intf{f}$ and these quantities are finite; then a minimizer of $\Intf{f}$ on $B$ will be called a $p$-robust  minimizer on $B$. A point $x$ will be called a $p$-robust local minimizer of $\Intf{f}$ provided the existence of some $\eta >0$ such that $x$ is a $p$-robust minimizer on $\mathbb{B}(x,\eta)$.
\end{definition}
The above definition is given only for finite measures, due to the fact that whenever $\mu(T)=+\infty$, we have that $\int\limits_T \| x(t) -y\|^p d\mu(t) =\infty$ for all $(x,y) \in \textnormal{L}^p(T,X)\times X\backslash \{ 0\}$. However, in many of the results, when we work with a general $\sigma$-finite measure, we can modify the measure space to work with an equivalent finite measure.

It is worth mentioning  that one can easily prove  (using \emph{H\"{o}lder's inequality} \cite[Corollary 2.11.5]{MR2267655})  that   a $p$-robust minimizer of $\Intf{f}$ is also an $r$-robust minimizer for every $r \geq p$.

 The following result gives sufficient conditions for $p$-robustness. Particularly, in the finite-dimensional setting, each local minimum is a $p$-robust minimum.
\begin{proposition}
	\label{Prop:SuficientCondition}\hspace{-0.2cm} Let $({\T},\mathcal{A},\mu)$ be a finite measure space. Consider $p \in [1,+\infty)$,  $q:= p/(p-1)$, and $B\subseteq X$
	such that $\dom \Intf{f} \cap B \neq \emptyset$. Suppose one of the following
	conditions is satisfied: 
	
	\begin{enumerate}[label={(\alph*)},ref={(\alph*)}] 
		
		\item \label{Prop:SuficientConditiona} For almost every $t\in {T}$, $%
		f(t,\cdot)$ is $\tau$-lsc, $B$ is $\tau$-closed and there exists $A\in 
		\mathcal{A}$ with $\mu(A) >0$ such that for all $t \in A$, $f(t,\cdot)$ is
		sequentially $\tau$-inf-compact, with $\tau$ being some topology, which is coarser than the
		norm topology (i.e. $\tau\subseteq \tau_{\|\cdot\|}$). 
		\item For almost every $t\in {T}$, $%
		f(t,\cdot)$ is $\tau$-lsc, and $B$ is sequentially $\tau$-inf-compact.
		\item \label{Prop:SuficientConditionb} For almost every $t\in {T}$, $f_t$ is
		Lipschitz on $X$ with some $q$-integrable constant. 
	\end{enumerate}
	
	Then 
	\begin{equation*}
		\wedge_{p,B} \Intf{f} =\inf_B \Intf{f}.
	\end{equation*}
\end{proposition}
\begin{proof}
	See \cref{proof_of_Prop:SuficientCondition}.
	\end{proof}

Now we present a fuzzy necessary condition for the existence of  a  $p$-robust  minimum in terms of the subdifferential of the data function $f$. The proof of this result is based in the application of \emph{Borwein-Preiss variational principle}  to  some appropriate function defined in $X \times \textnormal{L}^p(T,X)$.
\begin{theorem}
	\label{aproximatecalculsrules} Let  $p,q \in (1,+\infty)$ with $1/p + 1/q=1$. Assume that the
	measure $\mu$ is finite and that $x_0 \in X$ is a $p$-robust local minimizer of $%
	\Intf{f}$. Then there are sequences $y_n \in X$, $x_n \in \textnormal{L}^p({T}%
	,X)$, and $x_n^* \in \textnormal{L}^q({T},X^*)$ such
	that:  
		\begin{enumerate}[label={(\alph*)},ref={(\alph*)}]
			\item $x_n^*(t) \in \Asub f(t,x_n(t))$ ae,
			\item $\| x_0 - y_n\| \to 0$, $ \| x_0 - x_n(\cdot)\|_p  \to 0$,
			\item $\| x_n^*(\cdot)\|_q  \| x_n(\cdot)- y_n\|_p \to 0$,
			\item $\| \displaystyle \int\limits_T x_n^*(t) d\mu(t)\| \to 0$,
			\item $ \displaystyle \int\limits_T | f(t,x_n(t)) - f(t,x_0)|  d\mu(t)\to 0$. %
		\end{enumerate}
\end{theorem}
\begin{proof}
	See \cref{Apendix:aproximatecalculsrules}.
	\end{proof}

Now we establish the two main results of this section. In order to show how
to adapt some of the settings available in the literature to our framework,
we consider in the following theorems two normal integrands $f,g:{T}%
\times X\to \overline{\mathbb{R}}$, satisfying the following properties:
\begin{align}\label{P1P2} \tag{P}
\begin{array}{l}
	\text{For all } t\in {T} \text{ and all  } x\in X,\;\; f(t,x)\geq g(t,x).\\
	\text{The functions } g_t \text{ are }C^1 \text{ for all } t\in {T}.
\end{array}
\end{align}

A typical example  is $g(t,x)=\langle a^*(t),x \rangle + \alpha(t)$ with $a^* \in
\textnormal{L}^p({T},X^*)$ and $\alpha \in \textnormal{L}^1({T},\mathbb{R})$, which corresponds to the convex case; i.e., when $f(t,\cdot)$ is convex ae (see 
\cite{INTECONV,MR2291564,MR2444461}).
The first main result shows a fuzzy calculus rules for the Fr\'echet subdifferential. This calculus rule can be obtained using an appropriated modification of the measure space and the normal integrand function in order to satisfy the assumptions of \cref{Prop:SuficientCondition}, and then we  transform a local minimum into a $p$-robust minimum. The proof of this result shows in particular the high potential of our general setting.

\begin{theorem}
	\label{theoremsubdiferential}  \hspace{-0.2cm} Let $f,g$ be two normal integrands satisfying %
	\cref{P1P2} and $p,q\in(1,+\infty)$ with $1/p+1/q=1$. Assume that $\mu$ is finite and the function $\sup_{u\in X}\|
	\nabla g(\cdot,u)\|$ belongs to $\textnormal{L}^q({T},\mathbb{R})$.  Then for every  $x^* \in \Asub \Intf{f}(x)$ and every $w^*$-continuous seminorm $\rho$ in $X^*$, there exist
	sequences $y_n \in X$, $x_n \in \textnormal{L}^p({T},X)$, $x_n^* \in 
	\textnormal{L}^q({T},X^*)$ such that:  
		\begin{enumerate}[label={(\alph*)},ref={(\alph*)}]
			\item $x_n^*(t) \in \Asub f(t,x_n(t))$ ae.
			\item $\| x- y_n\| \to 0$,$ \| x - x_n(\cdot)\|_p \to 0 $,
			\item $\| x_n^*(\cdot)\|_q \| x_n(\cdot)- y_n\|_p \to 0$,
			\item $\displaystyle\int_{T} \langle x_n^*(t), x_n(t)- x\rangle d\mu(t)\to 0$,
			\item \label{theoremsubdiferentiald} $\rho\left( \displaystyle \int_{T} x_n^*(t) d\mu(t) - x^*\right) \to 0$,
			\item $ \displaystyle \int_{T}| f(t,x_n(t)) - f(t,x)| d\mu(t)\to 0$.
		\end{enumerate}
\end{theorem}

\begin{proof}
	First, assume that $g=0$. Then consider   $\varepsilon >0$ and  $\{ e_i\}_{i=1,...,k}$  a finite family of points such that $\rho(\cdot)=\max\{ \langle \cdot, e_i  \rangle : i=1,..,k \}$, and denote by $L:=\spn\{ x, e_i\}_{i=1}^{k}$ and  $K=L\cap \mathbb{B}(x,1)$. Since $x^\ast \in \Asub \Intf{f}(x)$ there are a ball  $\mathbb{B}(x,\eta)$  and $C^1$ function   $\phi : \mathbb{B}(x,\eta) \to \R$ such that $\nabla \phi(x)=x^*$ and $\Intf{f} -\phi$ attains  a local minimum at $x$. 
	
	Let us consider the measure space $(\tilde{{T}},\tilde{\mathcal{A}},\tilde{\mu}) $, where $\tilde{{T}}={T}\cup\{ \omega_1,\omega_2\}$ (with $\omega_1,\omega_2 \notin T$), $\tilde{\mathcal{A}}=\sigma(\mathcal{A}, \{ \omega_1\},\{\omega_2\})$ and $\tilde{\mu}(A)=\mu(A\backslash\{ \omega_1,\omega_2\}) + \1_{A}(\omega_1)+\1_{A}(\omega_2)$, together with the integrand function
	\begin{align}
	\tilde{f}(t,x)= \Bigg\{ \begin{array}{cl}
f(t,x), &\text{ if } x \in X \text{ and } t\in T,\\
-\phi(x),&\text{ if } x \in X \text{ and } t=\omega_1,\\
\delta_{K} (x),&\text{ if } x\in X \text{ and }t=\omega_2.
	\end{array}
	\end{align}
 Now  condition \cref{Prop:SuficientConditiona}  of   \cref{Prop:SuficientCondition} holds for the integrand $\tilde{f}$ and the  measure space $(\tilde{{T}},\tilde{\mathcal{A}},\tilde{\mu}) $. Furthermore, $\Intf{\tilde{f}}$ attains its minimum at $x$, and by \cref{Prop:SuficientCondition} we have that $x$ is a $p$-robust minimizer of $\Intf{\tilde{f}}$. Whence, by  \cref{aproximatecalculsrules} there exist  sequences $\tilde{y}_n \in X$, $\tilde{x}_n \in \textnormal{L}^p(\tilde{{T}},X)$, $\tilde{x}_n^* \in \textnormal{L}^q(\tilde{{T}},X^*)$ (with $1/p + 1/q=1$) such that:
		\begin{enumerate}[label=(\arabic*),ref=(\arabic*)]
			\item $\tilde{x}_n^*(t) \in \Asub \tilde{f}(t,\tilde{x}_n(t))$ ae,
			\item $\| x - \tilde{y}_n\| \to 0$, $\displaystyle \int\limits_{\tilde{T}} \| x - \tilde{x}_n(t)\|^p d\tilde{\mu}(t) \to 0 $,
			\item $ \| \tilde{x}_n^*(\cdot)\|_q \| \tilde{x}_n(\cdot)- \tilde{y}_n\|_p \to 0$,
			\item $\| \displaystyle \int\limits_{\tilde{T}} \tilde{x}_n^*(t) d\tilde{\mu}(t)\| \to 0$,
			\item $ \displaystyle \int\limits_{\tilde{T}} | f(t,\tilde{x}_n(t)) - f(t,x)|  d\tilde{\mu}(t)\to 0$.
		\end{enumerate}
	In particular, $\int_{\tilde{{\T}}} \tilde{x}^*_n(t) d\tilde{\mu}(t)$ is bounded, and so $\langle  \int_{\tilde{{\T}}}\tilde{x}^*_n(t)d\tilde{\mu}(t), \tilde{y}_n- x \rangle \to 0$. Hence, 
	\begin{align*}
	\bigg| \int_{\tilde{{\T}}}\langle \tilde{x}^*_n(t), \tilde{x}_n(t)- x\rangle d\tilde{\mu}(t)\bigg| &\leq   \bigg| \int\limits_{\tilde{{\T}}}\langle \tilde{x}^*_n(t), \tilde{y}_n - x\rangle d\tilde{\mu}(t) \bigg| \\& +  \int\limits_{\tilde{{\T}}} \| \tilde{x}_n^*(t)\| \| \tilde{x}_n(t)- \tilde{y}_n\| d\tilde{\mu}(t) \to 0.
	\end{align*}
	Now,  define $x_n(t):=\tilde{x}_n(t)$,  $x^*_n(t):=\tilde{x}_n^*(t)$ with $t\in {\T}$ and $y_n=\tilde{y}_n$. So, $x_n \in \textnormal{L}^p({\T},X)$, $x_n^* \in \textnormal{L}^q({\T},X^*)$, $x_n^*(t) \in \Asub f(t,x_n(t))$ ae, $\| x - y_n\| \to 0$, $ \int_{\T} \| x - x(t)\|^p d\mu(t) \to 0 $, $ \| x_n^*(\cdot)\|_q \| x_n(\cdot)- y_n\|_p \to 0$ and $  \int_{\T} | f(t,x_n(t)) - f(t,x)|  d\mu(t)\to 0$.
	
	Next, we get $\tilde{x}^*_n(\omega_1)=-\nabla \phi(\tilde{x}_n(\omega_1)) \overset{w^*}{\to} -\nabla\phi(x)=-x^*$. By using the convexity of $K$, we have that for a large enough $n$, $\tilde{x}^*_n(\omega_2) \in N_{K}(\tilde{x}_n(\omega_2 ))=L^\perp$. Therefore, 
	\begin{align*} 
		\rho\left( \displaystyle \int\limits_T x_n^*(t) d\mu(t)  - x^*\right)&\leq \rho\left( \displaystyle \int\limits_T x_n^*(t) d\mu(t)  +\tilde{x}^*_n(\omega_1) +\tilde{x}^*_n(\omega_2)\right)  \\&+ \rho( -\tilde{x}^*_n(\omega_1) - x^*) + \rho( \tilde{x}^*_n(\omega_2) ) \\
		&\leq \| \displaystyle \int\limits_T \tilde{x}_n^*(t) d\tilde{\mu}(t) \| + \rho( -\tilde{x}^*_n(\omega_1) - x^*) + \rho( \tilde{x}^*_n(\omega_2) )  \to 0.
	\end{align*}
	On the one hand, since $\tilde{x}^*_n(\omega_1)$ is bounded and  $\tilde{x}_n(\omega_1) \to x$, we have $\langle \tilde{x}^*_n(\omega_1), \tilde{x}_n(\omega_1) - x \rangle \to 0$. On the other hand, since for large enough $n$, $\langle \tilde{x}^*_n(\omega_2), \tilde{x}_n(\omega_2) - x \rangle = 0$, we get
	\begin{align*}
		\bigg| \int_{T} \langle x_n^*(t), x_n(t)- x\rangle d\mu(t)\bigg|   \leq&  \bigg|  \int_{\tilde{{T}}}\langle \tilde{x}^*_n(t), \tilde{x}_n(t)- x\rangle d\tilde{\mu}(t)\bigg|  + \bigg|   \langle \tilde{x}^*_n(\omega_1), \tilde{x}_n(\omega_1) - x \rangle \bigg| \\
		& + \bigg|  \langle \tilde{x}^*_n(\omega_2), \tilde{x}_n(\omega_2) - x \rangle   \bigg|   \to 0.
	\end{align*}
	Finally, if $g$ is not zero, we know  by  \cref{lemmadifferenciabilidadnorma} that the gradient of  $\Intf{g}$ is given by $\int_{T} \nabla g_t(x) d\mu(t)$. Then we apply the result to the integrand function $h:=f-g$, with the gradient $y^* :=x^* - \int_{T} \nabla g_t(x) d\mu(t) \in \Asub \Intf{h}(x)$, and the result follows after some standard calculations.
\end{proof}
The next theorem corresponds to the $p=+\infty$ version of \cref{theoremsubdiferential}.  This theorem is obtained using \cref{theoremsubdiferential}, and modifying the measurable selection in a set of small measure. It is important to mention that this technique produces measurable selections in spaces of functions, which are not necessarily Asplund spaces (like $\textnormal{L}^\infty({T},X)$ and $\textnormal{L}^1({T},X^*)$). Consequently, it would not be possible to get this fuzzy calculus  using simply the chain rule for the Fr\'echet subdifferential, as it was done in \cite{MR2444461} for the convex subdifferential.

\begin{theorem}
	\label{theoremsubdiferentialinfinity} Let $f,g$ be two normal integrands
	satisfying \cref{P1P2}. Assume that  the function $\sup_{u\in X}\| \nabla g(\cdot,u)\| $ belongs  to $\textnormal{L}^1({T},%
	\mathbb{R})$. Then for every  $x^* \in \Asub \Intf{f}(x)$ and  every $w^*$-continuous seminorm $\rho$ in $X^*$, there exist
	sequences $y_n \in X$, $x_n \in \textnormal{L}^\infty({T},X)$, $x_n^* \in \textnormal{L}^1({T},X^*)$ such that 
		\begin{enumerate}[label={(\alph*)},ref={(\alph*)}]
			\item $x_n^*(t) \in \Asub f(t,x_n(t))$ ae.
			\item $\| x- y_n\| \to 0$, $ \| x - x_n(\cdot)\|_\infty \to 0 $.
			\item $ \displaystyle\int_{T} \| x_n^*(t)\| \| x_n(t)- y_n\| d\mu(t) \to 0$.
			\item $\displaystyle\int_{T} \langle x_n^*(t), x_n(t)- x\rangle d\mu(t)\to 0$.
			\item \label{theoremsubdiferentialinfinityad} $\rho\big( \displaystyle \int_{T} x_n^*(t) d\mu(t) - x^*\big) \to 0$.
			\item $ \displaystyle \int_{T}| f(t,x_n(t)) - f(t,x)| d\mu(t)\to 0$.
		\end{enumerate}
\end{theorem}
\begin{proof}
	Consider  $\rho$ and $x^*\in \Asub \Intf{f}(x)$ as in the statement. First we  assume  that   $\mu$ is finite  and $g=0$,  and so  we have that $f(t,x)\geq 0$ for all $t\in {T}$ and  all $x\in X$. Let $\varepsilon\in (0,1)$ and define  $\tilde{f}(t,x'):= f(t,x') + \delta_{\mathbb{B}(x,\varepsilon)}(x')$. It follows that $x^*\in \Asub \Intf{\tilde{f}}(x)$. Then by  \cref{theoremsubdiferential} there exist measurable functions $ \tilde{x}_n\in \textnormal{L}^2({T},X)$, $ \tilde{x}_n^* \in \textnormal{L}^2({T},X^*)$ such that:
		\begin{enumerate}[label=(\arabic*),ref=(\arabic*)]
			\item $\tilde{x}_n^*(t) \in \Asub \tilde{f}(t,\tilde{x}_n(t))$ ae,
			\item $\| x - \tilde{y}_n\| \to 0$,$\displaystyle \int\limits_T \| x - \tilde{x_n}(t)\|^2 d\mu(t) \to 0 $,
			\item $ \| \tilde{x}_n^*(\cdot)\|_2 \| \tilde{x}_n(\cdot)- \tilde{y}_n\|_2 $,
			\item $\rho( \displaystyle \int\limits_T \tilde{x}_n^*(t) d\mu(t)  - x^*) \to 0$,
			\item $ \displaystyle \int\limits_T | f(t,\tilde{x}_n(t)) - f(t,x)|  d\mu(t)\to 0$.
		\end{enumerate}
	It is easy to see that if $\| \tilde{x}_n(t) -x\| <\varepsilon$, then $\tilde{x}_n^*(t)\in \Asub f(t,\tilde{x}_n(t))$. Define the measurable sets  $A_n:=\{ t \in {T}: \| \tilde{x}_n(t) -x\| =\varepsilon  \}$. The convergence in $\textnormal{L}^2({T},X)$ implies that $\mu(A_n)\to 0$. We take $n\in \N$ such that 
	\begin{enumerate}[label=(\roman*),ref=(\roman*)]
		\item $\| x -\tilde{y}_n\| \leq \varepsilon/2$, $\rho( \displaystyle\int\limits_T \tilde{x}_n^*(t) d\mu(t)  - x^* )\leq \varepsilon/3$,
		\item\label{ITEM2} $\displaystyle\int_{T}\| \tilde{x}^*_n(t)\| \| \tilde{y}_n - \tilde{x}_n(t)\| d\mu(t) \leq \varepsilon^2/6$, $\displaystyle\int_{\T} | f(t,\tilde{x}_n(t)) - f(t,x)|  d\tilde{\mu}(t)\leq \varepsilon/2 $,
		\item  $\displaystyle\int_{\T} \langle x_n^*(t), x_n(t)- x\rangle d\mu(t) \leq \varepsilon/3$ and  $\displaystyle\int_{A_n} f(t,x) d\mu(t) \leq \varepsilon^2/24.$
	\end{enumerate}
	 It follows from \cref{ITEM2} and the definition of $A_n$ that
	\begin{align*}
	\frac{\varepsilon^2}{6} &\geq \int_{\T}\| \tilde{x}^*_n(t)\| \| \tilde{y}_n - \tilde{x}_n(t)\| d\mu(t) \geq \int_{A_n}\| \tilde{x}^*_n(t)\| \| \tilde{y}_n - \tilde{x}_n(t)\| d\mu(t)\\
	&\geq \int_{A_n}\| \tilde{x}^*_n(t)\|\big(  \| x-  \tilde{x}_n(t)\|  - \| x- \tilde{y}_n \|\big)  d\mu(t)\geq  \int_{A_n}\| \tilde{x}^*_n(t)\|\big(\varepsilon  -  \frac{\varepsilon}{2} \big)  d\mu(t)\\
	&\geq \frac{\varepsilon}{2}  \int_{A_n}\| \tilde{x}^*_n(t)\|d\mu(t).
	\end{align*}
	Therefore, $ \int_{A_n}\| \tilde{x}^*_n(t)\|d\mu(t) \leq \frac{\varepsilon}{3}$. We set  $\varepsilon(t):=f(t,x)$ and by the nonnegativity  of the integrand, we have that $x$ is a $\varepsilon(t)$-minimum of $f(t,\cdot)$ for almost all $t\in A_n$. Then by  \cref{lemmasubdeltaminimun} there exist measurable functions $(y(t),y^*(t) )\in X\times X^*$ such that for almost all $t\in A_n$, $y^*(t)\in \Asub f(t,y(t))$, $\| y(t) - x\| \leq \varepsilon/2$, $|f(t,y(t)) - f(t,x)| \leq \varepsilon(t)$ and $\| y^*(t) \| \leq 8\varepsilon(t)/\varepsilon$. 
	
	Let us define $x(t):=\tilde{x}_n(t) \1_{A_n^c}(t) + y(t)\1_{A_n}$ and $x^*(t):=\tilde{x}^*_n(t) \1_{A_n^c}(t) + y^*(t)\1_{A_n}$.  Hence, $x^*(t)\in \Asub f(t,x(t))$ ae, $\| x- x(\cdot)\|_{\infty}\leq \varepsilon$, 
	\begin{align*}   
		\int_{T} \| y^*(t) \| d\mu(t) &= \int_{A_n^c}\| \tilde{x}_n^*(t) \| d\mu(t) + \int_{A_n}\| y^*(t) \| d\mu(t)\\
		& \leq \mu({T})^{1/2} \| \tilde{x}_n^*(\cdot) \|_2 +  \varepsilon/3,\\ 
		\int_{T} | f(t,x(t)) - f(t,x)|  d\mu(t)
		&= \int_{A_n^c} | f(t,\tilde{x}_n(t)) - f(t,x)|  d\mu(t) \\ & + \int_{A_n} | f(t,y(t)) - f(t,x)|  d\mu(t)\leq  \varepsilon/2 + \varepsilon/2  = \varepsilon.
	\end{align*}
	Furthermore, 
	\begin{align*}
		\rho(  \int_{T} x^*(t) d\mu(t)  - x^*) &\leq \rho(  \int_{T} \tilde{x}_n^*(t) d\mu(t)  - x^*)  +  \int_{A_n} \| \tilde{x}_n^*(t)\| d\mu(t) \\ & + \int_{A_n} \|y^*(t)\| d\mu(t) \leq \varepsilon/3 +\varepsilon/3 +\varepsilon/3 = \varepsilon,
		\end{align*}
	 so that  
	\begin{align*}                
	\int\limits_T \| x^*(t)\| \| \tilde{y}_n -x(t)\| d\mu(t) &=\int\limits_{A_n^c} \|\tilde{x}^*(t) \| \|\tilde{y}_n - \tilde{x}_n(t) \| d\mu(t)\\& +  \int_{A_n} \| y^*(t) \|\big( \| \tilde{y}_n - x \| + \|y(t) - x \|  \big)  d\mu(t) \\
	&\leq   \varepsilon^2/6 + \varepsilon^2/3 + \varepsilon^2/6  \leq \varepsilon.
	\end{align*}
	Finally, 
	\begin{align*}
		\big| \int_{T} \langle x^*(t), x(t)- x\rangle d\mu(t)\big| \leq& \Bigg|\int_{A_n^c} \langle \tilde{x}_n^*(t), \tilde{x}_n(t)- x\rangle d\mu(t)  + \int_{A_n} \langle y^*(t), y(t)- x\rangle d\mu(t) \Bigg|\\
		\leq &\big|\int_{{T}} \langle \tilde{x}_n^*(t), \tilde{x}_n(t)- x\rangle d\mu(t) \big| \\&+ \int_{A_n} \| \tilde{x}_n^*(t)\| \cdot \| \tilde{x}_n(t)- x\|  d\mu(t) \\
		& + \int_{A_n} \| y^*(t)\| \cdot \|y(t)- x\| d\mu(t) \leq 2\varepsilon/3 
		+\epsilon^2/6 \leq  \varepsilon.
	\end{align*}
	Now, if $\mu$ is $\sigma$-finite, consider $\nu(\cdot) = \int_{\cdot} k(t) d\mu(t)$, where $k>0$ is integrable and consider the integrand $\tilde{f}(t,x)=f(t,x)/k(t)$. So, $\Intf{\tilde{f}}^{\nu} =\Intf{f}^\mu$, and then by applying the previous part  we easily get  the result. The general case, when $g$ is not zero, follows  the same arguments given in the proof of  \cref{theoremsubdiferential}.\end{proof}

\begin{remark}\label{REMARK:CONVERGENCENORM}
		It has not escaped to our notice that if  one of the conditions of  \cref{Prop:SuficientCondition} holds, then the convergence of $\int\limits_{\T} x^\ast_n (t)d\mu(t)$ to the subgradient $x^\ast$ in \cref{theoremsubdiferential,theoremsubdiferentialinfinity}, with respect to the seminorm $\rho$, can be changed by the convergence in norm topology. Indeed, if one of the conditions of  \cref{Prop:SuficientCondition} holds, then we proceed similarly as in the proof of \cref{theoremsubdiferential}, by taking simply $K=L=X$. So, the estimates follow similarly, but with the norm instead of the seminorm $\rho$.
	\end{remark}
To ilustrate our results  we compute sequential formulae for  series of lower semicontinuous functions using the measure space $(\mathbb{N},\mathcal{P}(\mathbb{N}))$. This class of functions has been recently studied in the convex case (see, e.g., \cite{MR3571567,INTECONV,INTECONV2}), motivated by some applications to  entropy minimization. Moreover, in this case we can apply techniques of separable reduction, and extend the results to an arbitrary Asplund space. The proof of the following result is written in  \cref{proofofseparablereduction}, for simplicity.
\begin{corollary}
	\label{separablereduction}  The statement of  \cref{theoremsubdiferential,theoremsubdiferentialinfinity} holds if we
	assume that $X$ is a non-separable Asplund  space and $(T,\mathcal{A})=(\mathbb{N},\mathcal{P}(%
	\mathbb{N}))$.
\end{corollary}
\begin{proof}
	See \cref{proofofseparablereduction}.
	\end{proof}

The final result corresponds to an extension of \cite[Corollary 1.2.1]%
{MR2444461} to the case $p=+\infty$, which characterizes the convex subdifferential of  the  integral functional  $\Intf{f}$, when the data is a convex normal integrand.
\begin{corollary}
	\label{corollary:convex}  In the setting of  \cref{theoremsubdiferentialinfinity}, assume that $f$ is a convex normal integrand
	(i.e.; $f_t$ is convex for all $t\in {T}$). Then one has $x^*\in \sub \Intf{f}(x)$
	if and only if there are nets $x_\nu \in \textnormal{L}^\infty({T},X)$ and $x_\nu^*\in \textnormal{L}^1({T},X^*)$ such that
	\begin{enumerate}[label=(\alph*)]
		\item   $x_\nu^*(t) \in \sub 	f(t,x_\nu(t))$ ae, 
		\item $\| x - x_\nu(\cdot)\|_\infty \to 0 $,
		\item $\displaystyle\int_{T}
		x_\nu^*(t) d\mu(t) \overset{w^*}{\to} x^* $,
		\item $\displaystyle\int_{{T}}\langle x^*_\nu(t),
		x_\nu(t)- x\rangle d\mu(t)\to 0$,
		\item $\displaystyle\int_{T}| f(t,x_\nu(t)) - f(t,x)|
		d\mu(t)\to 0$.
	\end{enumerate}
	  If the space $X$ is reflexive we can take sequences instead
	of nets and the convergence of $\displaystyle\int_{T} x_\nu^*(t) d\mu(t)$ will be in the 
	norm topology.
\end{corollary}

\begin{proof}
	 The construction of the net follows similar and classical arguments. Indeed, consider  $x^* \in \sub \Intf{f}(x)$. Then take $\mathcal{N}_{0}$ the neighborhood system of zero for the $w^*$-topology and consider the  set $\mathbb{A}:=\N \times\mathcal{N}_{0}$, ordered by  $(n_1, U_1) \leq (n_2, U_2)$ if and only if $n_1 \leq n_2$ and $U_2 \subseteq U_1$. 
	 
	 Then by   \cref{theoremsubdiferentialinfinity}  we have that for every $\nu=(n,U)$   there are  $x_\nu  \in \textnormal{L}^\infty({T},X)$ and $x^*_\nu \in \textnormal{L}^1({T},X^*)$ such that 
		\begin{enumerate}[label=(\arabic*),ref=(\arabic*)]
			\item $x_\nu^*(t) \in \sub f(t,x_\nu(t))$ ae,
			\item $ \| x - x_\nu(\cdot)\|_\infty \leq 1/n$,
			\item $ \int\limits_T \langle  x_\nu^*(t),  x_\nu(t)- x\rangle d\mu(t) \to 1/n$,
			\item $  \int x_\nu^*(t) d\mu(t) - x^* \in U$,
			\item $ \int_{T}| f(t,x_\nu(t)) - f(t,x)| d\mu(t)\to 0$.
		\end{enumerate}
	Hence, the net $(x_\nu,x_\nu^*)$ satisfies the required properties. Conversely, assume that the net $(x_\nu,x_\nu^*)$ satisfies  the above properties. Then  for all $y \in X$
	\begin{align*}
		\langle x^* , y - x \rangle \leq& 	\langle x^* - \int_{T} x^*_\nu(t) d\mu(t) , y - x \rangle  +  \int_{T}  \langle  x^*_\nu(t), y - x_\nu(t) \rangle d\mu(t)   \\&+\int_{T}  \langle  x^*_\nu(t), x_\nu(t) - x \rangle  d\mu(t)\\
		\leq& \langle x^* - \int_{T} x^*_\nu(t) d\mu(t) , y - x \rangle + \Intf{f}(y) - \int_{T} f(t,x_\nu(t))\\& + \int_{T}  \langle  x^*_\nu(t), x_\nu(t) - x \rangle  d\mu(t).
	\end{align*}
	So, taking the limits we conclude	$\langle x^* , y - x \rangle \leq \Intf{f}(y) -\Intf{f}(x)$, since $y$  is arbitrary we get the result.
	
	Finally, when  $X$ is reflexive, without loss of generality, we can assume that  criterion \cref{Prop:SuficientConditiona} of  \cref{Prop:SuficientCondition} is satisfied; otherwise, we take $\tilde{f}_t:=f_t + \delta_{\mathbb{B}(x,1)}$.  Then, by \cref{REMARK:CONVERGENCENORM}, we can  construct  a sequence with the desired property using the norm instead of a family of seminorms.
\end{proof}

\begin{remark}
	It is worth comparing the results given in \cite[Theorem 1.4.2]{MR2444461}
	with  \cref{corollary:convex}:
	Let us recall that a
	functional $\lambda ^{\ast }\in \textnormal{L}^{\infty }(T,X)^{\ast }$ is called singular
	if there exists a sequence of measurable sets $T_{n}$ such that $
	T_{n+1}\subseteq T_{n}$, $\mu (T_{n})\rightarrow 0$ as $n\rightarrow \infty $
	and $\lambda ^{\ast }(g\mathds{1}_{T^c_{n}})=0$ for every $g\in \textnormal{L}^{\infty
	}(T,X)$. The set of all singular elements is denoted by  $\textnormal{L}^{\textnormal{sing}}(T,X)$. It is well-known that the dual of  $
\textnormal{L}^{\infty }(T,X)$ can be  represented  as the direct sum of $\textnormal{L}^1(T,X^\ast)$ and $\textnormal{L}^{\textnormal{sing}}(T,X)$
(see, for example, \cite{MR0467310,0036-0279-30-2-R03}). 
	
	In \cite[Theorem 1.4.2]{MR2444461} the authors proved similar characterizations of the subdifferential of $\Intf{f}$, where they established that $x^*\in \sub \Intf{f}(x)$
	if and only if there are nets $x_\nu \in \textnormal{L}^\infty({T},X)$, $x_\nu^*\in \textnormal{L}^1({T},X^*)$  and $\lambda^\ast_\nu \in \textnormal{L}^{\textnormal{sing}}(T,X)$ such that
	
		\begin{enumerate}[label=(\alph*)]
		\item   $x_\nu^*(t) \in \sub 	f(t,x_\nu(t))$ ae, 
		\item $\| x - x_\nu(\cdot)\|_\infty \to 0 $,
		\item $\displaystyle\int_{T}
		x_\nu^*(t) d\mu(t) + A^\ast (\lambda^\ast_\nu) \overset{w^*}{\to} x^* $,
		\item $\displaystyle\int_{{T}}\langle x^*_\nu(t),
		x_\nu(t)- x\rangle d\mu(t) + \lambda^\ast_\nu ( x^*_\nu(\cdot ) -A(x)) \to 0$,
		\item $\displaystyle\int_{T}| f(t,x_\nu(t)) - f(t,x)|
		d\mu(t)\to 0$,
	\end{enumerate}
	where $A: X\to \textnormal{L}^\infty(T,X)$ is the linear functional given by $A(x):=x\1_T$ and $A^\ast$ denotes its adjoint. Furthermore, the functionals $\lambda_\nu $ belong to  the normal cone of $ \IntfLp{f}^\infty$ at the constant function $A(x)$.
	In other words, we have extended to the non-convex case this class of results by using the Fr\'echet subdifferential. Also, our characterizations  of the subdifferential of $\Intf{f}$ are tighter when the integrand is convex, since that we do not require the use of  \emph{singular elements} from the dual of $\textnormal{L}^{\infty}({T},X)$.
\end{remark}

\section{Limiting and Clarke-Rockafellar subdifferentials}

\label{Section:Limitingsub}  The aim of this section is to establish
upper-estimates for the limiting and Clarke-Rockafellar
subdifferentials at a point $x\in \dom \Intf{f}$, in terms of the corresponding
subdifferential of the data function $f_t$ at the same point. We will focus on the case when $X$ is a separable Asplund space.

In view of the  results of the last section, we need to ensure the boundedness of
the approximate sequences involved in the previous formulas of the subdifferential in order to establish upper-estimates, which are expressed at the exact point. So, the next part concerns criteria to guarantee this
property. For this reason, we introduce the following definitions that allow
us to extend the classical results, which generally consider some local Lipschitz
continuity property of the integral functional (see for instance \cite[Theorem 2.7.2]%
{MR1058436} or \cite{MR3783778}).

We introduce the concept of $w^*$-compact soles (see \cite[Proposition 2.1]%
{cornet2002fatou}).

\begin{definition}[Integrable compact sole]
	\label{Def:INTCOMPACT}  Consider a measurable multifunction $C: {T}
	\rightrightarrows X^*$ with non-empty closed values. \newline
	(i) We say that $C$ has
	an integrable compact sole if there exist $e \in X$ and $\gamma >0$
	such that for every measurable selection $c^*$ of $C$  
	\begin{align*}
		\gamma \langle c^*(t) , e\rangle \geq \|c^*(t) \| \;ae.
	\end{align*}
	(ii) We denote $$UI(C) :=\{ u \in X:\sigma_{C(\cdot)}(u)^{+} \in 
	\textnormal{L}^1(T,\mathbb{R})\},$$ where $\sigma_{C(t)}(u)^{+}:=\max\{ \sigma_{C(t)}(u),0  \}$.
\end{definition}
Basically, the set $UI(C) $ denotes the directions for which all the measurable selections are uniformly integrable.

In order to understand better this notion, we include in the Appendix a characterization of the integrable compact sole property in terms of the primal space  (see \cref{characterizationofcompactsole}).
\begin{theorem}
	\label{MordukhovichSublimitingaExact} Let $x \in \dom \Intf{f}$ and suppose there
	exist $\epsilon >0$, a measurable multifunction $C :{T} \rightrightarrows X^*$
	which has an integrable compact sole, and an integrable function $K(\cdot) >0$
	such that  
	\begin{align}  \label{eq:boundedsub}
		\Asub f(t, x^{\prime }) \subseteq K(t) \mathbb{B} + C(t), \forall
		x^{\prime }\in \mathbb{B}(x,\epsilon),\;\forall t\in {T}.
	\end{align}
	Then  
	\begin{align}
		\sub  \Intf{f}(x) \subseteq&\bigcap\bigg\{ \int\limits_T \sub f(t,x)d\mu(t) +
		UI(C)^{-}+ W^\perp\bigg\},  \label{EQ:1} \\
		\sub^{\infty}\Intf{f}(x) \subseteq& \bigcap\bigg\{ \int_{T} \sub^\infty f(t,x) d\mu(t) +UI(C)^{-}+ W^\perp\bigg\},  \label{EQ:2}
	\end{align}
	where the intersection is over all finite-dimensional subspaces $W \subseteq
	X$. Consequently,  
	\begin{align}\label{clarke}
		\sub_C \Intf{f}(x) \subseteq \cco^{w^*}\bigg\{ \int\limits_{{T}} \sub f(t,x)
		d\mu + \int_{T} \sub^\infty f(t,x) d\mu(t) +UI(C)^{-} \bigg\}.
	\end{align}
\end{theorem}

\begin{proof}
	Let $x^* \in \sub \Intf{f}(x)$ and $y^* \in \sub^\infty \Intf{f}(x)$. Consider a finite family of linearly independent points $\{ e_i\}_{i=1}^p$,  $W:=\spn \{ e_i\}$ and  $\rho(\cdot) :=\max\{ | \langle \cdot , e_i  \rangle | \}$. Then by the auxiliary result \cref{MordukhovichSub}, proved in \cref{Section_Fatou_type_MordukhovichSub}, there exist sequences $x_n,y_n \in \textnormal{L}^\infty({T},X)$, $x_n^*, y_n^*\in \textnormal{L}^1({T},X^*)$  and $\lambda_n \to 0^+$ such that:
	\begin{enumerate}[label={(\roman*)},ref={(\roman*)}]
			\item $x_n^*(t) \in \Asub f(t,x_n(t))$ ae,
			\item $\| x - x_n(\cdot)\|_\infty \to 0 $,
			\item $\rho( \int_{\T} x_n^*(t) d\mu(t) - x^*)  \to 0$,
			\item $ \lim \int_{\T} | f(t,x_n(t)) - f(t,x)|  d \mu(t)\to 0$,
		\end{enumerate}
	
		\begin{enumerate}[label={(\roman*$^\infty$)},ref={(\roman*$^\infty$)}]
			\item $y_n^*(t) \in \Asub f(t,y_n(t))$ ae,
			\item $\| x - y_n(\cdot)\|_\infty  \to 0 $,
			\item\label{conve:c} $\rho(\lambda_ n \cdot  \int_{\T} y_n^*(t) d\mu(t) - y^*) \to 0$,
			\item $ \lim \int_{\T} | f(t,y_n(t)) - f(t,x)|  d \mu(t)\to 0$.
		\end{enumerate}	
	Hence, (for large enough $n$) relation \cref{eq:boundedsub} implies that  $x_n^*(t) \in K(t)\mathbb{B} +C(t)$ and $y_n^*(t)\in K(t)\mathbb{B} + C(t)$ ae. Now, consider the multifunctions 
	\begin{align*}
		\mathcal{G}_1(t)&:=\{  (a^\ast,b^\ast) \in  K(t)\mathbb{B} \times  C(t) : 	x_n^\ast(t) = a^\ast + b^\ast	\},\\
		\mathcal{G}_2(t)&:=\{  (a^\ast,b^\ast) \in  K(t)\mathbb{B} \times  C(t) : 	y_n^\ast(t) = a^\ast + b^\ast	\},
		\end{align*}
	 which are graph measurable (see, e.g., \cite[Proposition 1.41, Proposition 1.43 and Remark 1.44]{MR1485775}).	Therefore, by the \emph{measurable selection theorem} (see \cref{measurableselectiontheorem})  there are measurable selections $h^1_n(t),h^2_n(t)  \in \mathbb{B}(0, K(t))$ and $c^1_n(t),c_n^2(t) \in C(t)$ such  that $x_n^*(t)=h^1_n(t) +  c^1_n(t)$ and $ y_n^*(t)=h^2_n(t) +  c^2_n(t)$. From the fact that  $C$ has an integrable compact sole, there exist $e \in X$ and $\gamma >0$ such that $\| c_n^i(t) \| \leq \gamma \langle c_n^i(t) ,e\rangle$ for $i=1,2$ and almost all $t\in T$. Then
	\begin{equation}
	\begin{aligned}\label{eq:UI}
	\displaystyle \int_{T} \| x^*_ n \| d\mu &\leq \int_{T} K d\mu + \gamma \int_{T} \langle c^1_n(t) ,e\rangle d\mu(t) \\&=  \int_{T} K(t) d\mu(t)  + \gamma\big( \int_{T} \langle x_n^*(t) ,e\rangle d\mu(t) - \int_{T} \langle h^1_n(t) ,e\rangle d\mu(t\big)\\
	&\leq (1+\| e\|)\int_{T} K(t) d\mu(t)  + \gamma \int_{T} \langle x_n^*(t) ,e \rangle d\mu(t).
	\end{aligned}
	\end{equation}
	So, assuming that $e\in W$, the sequence $(x^*_n)$ is bounded in $\textnormal{L}^1({T},X^*)$  and, obviously, the same holds for the sequence $(\lambda_n y_n^*)$.  Then, by  \cref{MordukhovichSub}, and observing that $C(x_n^\ast) \subseteq UI(C)^{-}$ and $C(\lambda_n y_n^\ast) \subseteq UI(C)^{-}$ (see the notation in \cref{MordukhovichSub}),
	$$x^* \in \displaystyle\int\limits_T \sub f(t,x)d\mu(t) + UI(C)^{-}+ W^\perp$$ and 
	$$y^* \in \displaystyle\int\limits_T \sub^\infty f(t,x)d\mu(t) + UI(C)^{-}+ W^\perp.$$	
	 Since $W$ was chosen arbitrary,   \cref{EQ:1,EQ:2} follow. Finally, \cref{clarke} follows from \cref{definition:clarke}.
\end{proof}

\begin{remark}
	When the measurable function $C$ has cone values, it is easy to see that \cref{eq:boundedsub} implies that for all $t\in {T}$, $\sub^\infty f(t,x) \subseteq C(t)$ and $UI(C)=\{ u \in X : u \in C^-(t) 
		\text{ ae} \}$. In addition, if the values of $C$ are also $w^*$-closed and
		convex, then the integrable compact sole property can be understood in terms
		of the negative polar set $C^-(t)$ (see  \cref{compactsoloiffnonemptyinterior}). The most simple case is when $C$
		is a fixed $w^*$-closed convex cone; in this case,  \cref{compactsoloiffnonemptyinterior} characterizes the compact sole property as an  interior non-emptiness condition of the polar cone $C^-(\subseteq X)$. In particular, when the cone $C(t)=C=\{ 0\}$, we have that  \cref{eq:boundedsub} implies that for almost all $t\in \T$ the function $f_t$ is  Lipschitz continuous on $\mathbb{B}(x,\epsilon)$ with constant $K(t)$ (see, e.g., \cite[Theorem 3.52]{MR2191744}), and consequently this  recovers the classical framework of a normal integrand which is Lipschitz continuous (see, e.g., \cite[Theorem 2.7.2]{MR1058436}).
	\end{remark}

The next result corresponds to the explicit case when the measurable function $C$ in \cref{eq:boundedsub} is a fixed $w^\ast$-closed convex cone.
\begin{corollary}
	\label{MordukhovichSublimitingaExactcorolary} In the setting of  \cref{MordukhovichSublimitingaExact}, we assume that the multifunction $C$ is a
	constant $w^\ast$-closed convex  cone. Then  
	\begin{align*}
		\sub \Intf{f}(x) \subseteq\bigcap\big\{ \int\limits_T \sub f(t,x)d\mu(t) +
		C+ W^\perp\big\};\; \text{ and }\; \sub^{\infty}\Intf{f}(x) \subseteq& C,
	\end{align*}
	where the intersection is over all finite-dimensional subspaces $W \subseteq
	X$. Consequently, 
	\begin{align*}
		\sub_C \Intf{f}(x) \subseteq \cco^{w^*}\bigg\{ \int\limits_{{T}} \sub f(t,x)
		d\mu + C\bigg\}.
	\end{align*}
\end{corollary}

\begin{proof}
	Let us check that  $UI(C)=C^-$. Indeed, since that $\sigma_{C} = \delta_{C^-}$, we have that  
	$\sigma_{C}(u)^+ \in \textnormal{L}^1(T,\R)$ if and only if $u \in C^{-}$, which means that $UI(C)=C^-$. Now, by the Bipolar theorem  (see, e.g., \cite[Theorem 3.38]{MR2766381}), we have that $UI(C)^{-}=C$. Finally, using \cref{MordukhovichSublimitingaExact} we get the result.
	\end{proof}

	The motivation for using the boundedness condition \cref{eq:boundedsub} comes
	from  applications to stochastic programming; more precisely, applications
	to probability constraints (see \cite{MR3273343,MR3594331,Hantoute2018}), where
	the authors impose boundedness conditions over the gradients of the involved
	functions to guarantee the interchange between the sign of the integral and
	the subdifferential.

The following examples show the importance of using the multifunction $C$ in  \cref{MordukhovichSublimitingaExact,MordukhovichSublimitingaExactcorolary}.
\begin{example}
	\label{example:2} Consider the integrand $f:]0,1] \times \mathbb{R}\to
	[0,+\infty)$ given by  
	\begin{equation*}
		f(t,x)=\bigg\{ 
		\begin{array}{ccc}
			x^{3/2} t^{-1+x} & if & x > 0, \\ 
			0 & if & not.%
		\end{array}%
	\end{equation*}
	It is easy to check that $f$ is continuously differentiable with respect to $%
	x$ and  
	\begin{equation*}
		\Intf{f}(x)=\bigg\{ 
		\begin{array}{ccc}
			\sqrt{x} & if & x > 0, \\ 
			0 & if & not.%
		\end{array}%
	\end{equation*}
	Then we easily get $\sub \Intf{f}(0)=[0,+\infty)$, $$\Asub f(t,x)=\bigg\{ 
	\begin{array}{ccc}
	\frac{3}{2} x^{1/2} t^{-1+ x}+ x^{3/2}\ln(t)t^{-1+x} & if & x > 0, \\ 
	0 & if & not,%
	\end{array}%
	$$ and $\sub f(t,x)=\{0\}$. Then we can consider $C=[0,+\infty)$, so that $$\sub %
	\Intf{f}(0)= \int_{]0,1]} \sub f_t(0)d\mu(t) + C =\{ 0\} +[0,+\infty).$$ The
	same example can be modified as 
	\begin{equation*}
		f(t,x)=\bigg\{ 
		\begin{array}{ccc}
			x^{2} t^{-1+x} & if & x > 0, \\ 
			0 & if & not.%
		\end{array}%
	\end{equation*}
	Then one has  
	\begin{equation*}
		\Intf{f}(x)=\bigg\{ 
		\begin{array}{ccc}
			x & if & x > 0, \\ 
			0 & if & not.%
		\end{array}%
	\end{equation*}
	So, the integral functional $\Intf{f}$ is Lipschitz continuous, but it is not true
	that $\sub \Intf{f}(0) =\{0,1\}$ is included in $\int_{]0,1]} \sub f(t,0
	)d\mu(t)=\{0\}$, as in classical results (see \cite[Lemma 6.18]{MR2191745}
	and also \cite{MR3783778} for an extension of this result). However,
	 \cref{MordukhovichSublimitingaExactcorolary} guarantees the
	inclusion $\sub \Intf{f}(0) \subseteq \int_{]0,1]} \sub f(t,0 )d\mu(t) +
	[0,+\infty)$. 
\end{example}

\begin{remark}
	As a final comment we recall that in the finite-dimensional setting two lsc
	functions $f_1,f_2$ satisfy the sum rule inclusion $\sub (f_1 + f_2)(x)
	\subseteq \sub f_1(x) + \sub f(x)$ at a point $x$ provided that the
	asymptotic qualification condition 
	\begin{align*}
	x_1^\ast \in \sub^\infty f_1(x), \; x_2^\ast \in \sub^\infty f_2(x)	 \text{ and }x_1^\ast + x_2^\ast=0 \Rightarrow x_1^\ast=x_2^\ast=0.
		\end{align*}
	holds (see, e.g., \cite{MR2191744,mordukhovich2018variational,MR1491362,MR1710152,MR2144010}). However, the reader can notice that in the above
	example the integrand is continuously differentiable, then the singular
	subdifferential $\sub^\infty f_t(0)=\{ 0\}$ for all $t\in {T}$. In other words, it is
	not possible to recover similar criteria, as in the finite sum, in terms of
	the singular subdifferentials, to get an inclusion of the form $\sub \Intf{f}(x)\subseteq
	\int_{{T}} \sub f_t(x)d\mu(t)$.
\end{remark}
The final result gives criteria for the Lipschitz continuity and differentiability of the function $\Intf{f}$.
\begin{corollary}\label{COROLLARYDIFFE}
	In the setting of  \cref{MordukhovichSublimitingaExactcorolary}, assume that the multifunction $C=\{0\}$. Then $\Intf{f}$ is locally Lipschitz around $x$. In addition, if $X$ is finite-dimensional and $\sub f(t,x')$ is single  valued  ae for all $x'$ in a neighborhood of $x$, then $\Intf{f}$ is continuous differentiable at $x$. 
\end{corollary} 
\begin{proof}
	By  \cref{MordukhovichSublimitingaExactcorolary}, the Clarke subdifferential  $\sub_C \Intf{f}$ is bounded by $M:=\int K(t) d\mu(t)$  in a neighborhood of $x$. Then a straightforward  application of Zagrodny's Mean Value Theorem (see, e.g., \cite[Theorem 4.3]{MR972409}, or  \cite[Theorem 3.52]{MR2191744}) shows that $\Intf{f}$ is Locally Lipschitz around $x$. Furthermore, if $X$ is finite-dimensional and $\sub  f(t,x')$ is single  valued  ae for all $x'$ in a neighborhood of $x$, then $\sub_C \Intf{f}$ is single-valued  for all $x'$ in a neighborhood of $x$, and so \cite[Proposition 2.2.4 and its Corollary]{MR1058436} imply the result.
\end{proof}

\
\section*{Concluding Remarks}
 In this paper, we gave new and explicit formulae for the subdifferential of non-necessarily convex integrals, that are defined on infinite-dimensional Banach spaces. The resulting formulae are given exclusively by means of the corresponding subdifferentials of the integrand functions. All this analysis is done without requiring any qualification conditions.

\section*{Acknowledgments} The authors are  grateful to the  anonymous referees for their valuable remarks and suggestions that have greatly helped to improve this manuscript.

\appendix\label{APPENDIX}
	\section*{Appendix}
			Next, in the last part of this paper we recall some results and we prove differentiability properties of integral functions. We also include here some technical lemmas relying on variational principles, and give necessary conditions for the existence of $p$-robust minima.

			\section{Continuity and differentiability of integral functionals}
			
				We shall need the following lemma, which shows that the convergence of the values of the integral functional implies a stronger  convergence of the values of the data. This result has been proved in  \cite[Lemma 37]{MR3767752} (see also \cite{MR1618939}), but for the sake of completeness we present a proof.
			\begin{lemma}\label{lemma:convergenciaL1}
				Consider $x_n \in \textnormal{L}^p({\T},X)$ such that $x_n \overset{ \textnormal{L}^p}{\to } x$ and  $$\lim\int\limits_T f(t,x_n(t)) d \mu(t)= \int\limits_T f(t,x(t)) d \mu(t)\in\R.$$  Then $	\displaystyle \lim \int\limits_T | f(t,x_n(t)) - f(t,x(t))|  d \mu(t)= 0$.
			\end{lemma}
			\begin{proof}
				Fix $\delta >0$. Hence, by the lower semicontinuity of $\IntfLp{f}$ in $\textnormal{L}^p({\T},X)$ there exists $\varepsilon  >0$ such that $-\delta/4 + \hat{\Intf{f}}(x) \leq \IntfLp{f}(y(\cdot))$ for every $y \in \mathbb{B}_{\textnormal{L}^p ({\T},X)}(x,\varepsilon)$. Since  $x_n \to x$, there exists $n_1 \in \N$ such that $x_n \in \mathbb{B}_{\textnormal{L}^p({\T},X)}(x,\varepsilon)$ for every $n\geq n_1$. In particular, for every $A\in \mathcal{A}$  and every $n \geq n_1$ the function  $y:=x_n \1_{A} + x\1_{A^c} \in \mathbb{B}_{\textnormal{L}^p({\T},X)}(x,\varepsilon)$, and then   $-\delta/4 + \int_{A}f(t,x(t)) d\mu(t) \leq \int_{A}f(t,x_n(t)) d\mu(t) $ for every $A\in \mathcal{A}$. This yields, for all $A \in \mathcal{A}, $ and all $n \geq n_1$,
				\begin{align*}
					-\delta/4 + \int\limits_{A}f(t,x(t)) d\mu(t) &\leq\int\limits_{A}f(t,x_n(t)) d\mu(t) \\&=  \int\limits_{{\T}}f(t,x_n(t)) d\mu(t)  -  \int\limits_{A^c}f(t,x_n(t)) d\mu(t) \\
					&\leq  \int\limits_{{\T}}f(t,x_n(t)) d\mu(t) - \int\limits_{A^c}f(t,x(t)) + \delta/4.
				\end{align*}
				From the fact that $ \lim\int_{\T} f(t,x_n(t)) d \mu(t)= \int_{\T} f(t,x(t)) d \mu(t)$ there  exist  $n_2 \geq n_1$ such that  $\int_{\T} f(t,x_n(t)) d \mu(t) \leq \int_{\T} f(t,x(t)) d \mu(t) + \delta/4$ for all $n \geq n_2$. Thus, for all $A\in \mathcal{A}$ and all $ n \geq n_2$ 
				\begin{align*}
					-\delta/4 + \int\limits_{A}f(t,x(t)) d\mu(t) &\leq \int\limits_{A}f(t,x_n(t)) d\mu(t)\\ &\leq  \int\limits_{{\T}}f(t,x_n(t)) d\mu(t) - \int\limits_{A^c}f(t,x(t)) + \delta/4\\
					&\leq \int\limits_T f(t,x(t)) d \mu(t) +\delta/4- \int\limits_{A^c}f(t,x(t))  + \delta/4 \\
					&=  \int\limits_{A}f(t,x(t)) d\mu(t)  +\delta/2.
				\end{align*}
				Then, considering the measurable sets $A^+_n:=\{ t \in {\T} :  f(t,x_n(t)) - f(t,x(t)) >0     \}$ and $A^-_{n}:=\{ t \in {\T} :      f(t,x_n(t)) - f(t,x(t)) <0  \}$, we get
				\begin{align*}
					\displaystyle \int\limits_T|  f(t,x_n(t)) - f(t,x(t))| d \mu(t) &= \int\limits_{A^+_n}  f(t,x_n(t)) - f(t,x(t)) d \mu(t) &\\& + \int\limits_{A^-_n}  f(t,x(t)) - f(t,x_n(t)) d \mu(t) \\
					&\leq \delta/2 + \delta/4< \delta;
				\end{align*}
				that is,  $\displaystyle \int\limits_T|  f(t,x_n(t)) - f(t,x(t))| d \mu(t)  \to 0$.
			\end{proof}
			
			The following lemma is a simple application of classical rules concerning  differentiation of integral functionals. 
			\begin{lemma}\label{lemmadifferenciabilidadnorma}
				Let $\mu$  be a finite  measure and let $f:T\times X \to \Rex$  be a normal integrand  Lipschitz on $\mathbb{B}(x_0,\gamma)$  with some $p$-integrable constant, that is to say, there exists   $K \in \textnormal{L}^p({\T},\R)$ such that $|f(t,x) - f(t,y)| \leq K(t)|x-y|$, for all $x,y \in \mathbb{B}(x_0,\gamma)$ and all $t \in {\T}$. Assume that  the functions  $f_t$ are Fr\'echet differentiable  at $x_0$ ae. Then $\Intf{f}$ is Fr\'echet differentiable at $x_0$, $\nabla f(\cdot ,x_0)$ belongs to $\textnormal{L}^p({\T},X^*)$ and $\nabla \Intf{f}(x_0) =\int_{\T} \nabla f_t(x_0) d\mu(t)$. Moreover, if $f_t$ are $C^1$ on $\inte(\mathbb{B}(x_0,\gamma ))$, then $\Intf{f}$ is $C^1$ on  $\inte (\mathbb{B}(x_0,\gamma))$.
			\end{lemma}
			\begin{proof}
				First, the measurability and the integrability  of the function $t \to \nabla f_t(x_0)$ follows from the fact that for every $h\in X$, $\langle \nabla f_t(x_0), h \rangle =\lim\limits_{s\to 0^+} \frac{f(t, x_0 + sh) -f(t,x_0)}{s}$ and $\| \nabla f_t (x_0)\| \leq K(t)$ (see, e.g., \cite[\S2.1 Theorem 2 and \S2.2 Theorem 2]{MR0453964}). Now, take any sequence $(0,\gamma)\ni s_n \to 0^+$. Since $\mathbb{B}$ is bounded we can assume that $x_0 + s_n h \in \mathbb{B}(x_0,\gamma)$ for every $n \in \N$ and $h \in \mathbb{B}$, so that when the space $X$ is separable, the measurability of  
				$$t \to \sup\limits_{h \in \mathbb{B}}  \bigg| \frac{f_t(x_0 + s_n h) - f_t(x_0)}{s_n} - \langle \nabla f_t(x_0), h \rangle \bigg|$$
				 follows from the Lipschitz continuity of the integrand and the separability of $\mathbb{B}$. We notice that this function is bounded from  above by $K$; moreover, it converges to zero (ae) as $n \to \infty$.  Then  by  Lebesgue's dominated convergence theorem (see, e.g., \cite[Theorem 2.8.1]{MR2267655}) we get $$\lim\limits_{n\to \infty}\sup_{h \in \mathbb{B}} \bigg|\frac{ \Intf{f}(x_0 +s_n h) - \Intf{f}(x_0)}{s_n} - \int_{\T} \langle \nabla f_t(x_0), h \rangle d\mu(t) \bigg|\overset{n\to \infty }{\rightarrow}0,$$ which  concludes the first part.

				To prove the continuity of the derivative $\nabla \Intf{f} : \inte(\mathbb{B}(x_0,\gamma)) \to (X^*,\| \cdot \|)$, consider $x_n \to x\in \inte(\mathbb{B}(x_0,\gamma)$ with $x_n \in \mathbb{B}(x_0,\gamma)$. Then for almost all $t \in {\T}$, $$\lim\limits_{n\to \infty}\Bigg|  \nabla f_t (x)-\nabla f_t(x_n)\Bigg|= 0,$$ and $$g_n(t):=\sup_{h \in \mathbb{B}} |\langle  \nabla f_t (x)-\nabla f_t(x_n),h \rangle | \leq 2K(t)\; ae.$$ Then, again by the Lebesgue dominated convergence theorem, we get $$\big|  \nabla \Intf{f}(x) - \nabla \Intf{f}(x_n) \big| \overset{ n \to \infty }{\longrightarrow}0.$$
			\end{proof}

		\section{Variational principles}

			Now, we recall  the Borwein-Preiss Variational Principle, for which we need to introduce the notion of type-gauge functions.
		\begin{definition}\cite[Definition 2.5.1]{MR2144010}
			Consider $(\mathcal{X},d)$ a metric space. We say that a continuous function $\rho : \mathcal{X} \times \mathcal{X} \to [0, +\infty]$ is a gauge-type function provided that:
			\begin{enumerate}[label=(\roman*)]
				\item $\rho(x,x)=0$ for all $x\in \mathcal{X}$,
				\item for any $\epsilon >0$, there exists $\eta >0$ such that for all $y,z \in \mathcal{X}$ we have $\rho(y,z)\leq \eta$ implies that $d(y,z ) \leq \epsilon$.
			\end{enumerate}
			\end{definition}	
			The next result corresponds to  the   Borwein-Preiss Variational Principle.
			\begin{proposition} \cite[Theorem 2.5.3]{MR2144010}\label{BorweinPreiss:VP}
				Let $(\mathcal{X},d)$ be a complete  metric space and let $f: \mathcal{X}\to  \mathbb{R}\cup\{+\infty\}$ be a lsc function bounded from below. Suppose that $\rho$ is a gauge-type function and $(\eta_i)_{i=0}^{\infty}$ is a sequence of positive numbers, and suppose that $\epsilon >0$ and $z\in \mathcal{X}$ satisfy
				\begin{align*}
					f(z) \leq \inf\limits_{ \mathcal{X}} f +\epsilon.
					\end{align*}  
				Then there exist $y$ and a sequence $(x_i)$ such that 
				\begin{enumerate}[label=(\roman*)]
					\item $\rho(z,y) \leq \epsilon/\eta_0$, $\rho(x_i,y) \leq \epsilon/(2^i \eta_0  )$ for all $i=1,2,...$.
					\item $f(y)  + \sum\limits_{i=0}^\infty \eta_i \rho(y,x_i)\leq  f(z)$,
					\item $f(x) + \sum\limits_{i=0}^\infty \eta_i \rho(x,x_i) > f(y) +  \sum\limits_{i=0}^\infty \eta_i \rho(y,x_i)$, for all $x \in \mathcal{X}\backslash \{ y \}$.
					
 				\end{enumerate}
				\end{proposition}

				We recall that $(X,\| \cdot \|)$ is assumed to be a separable Asplund space and its norm is Fr\'echet differentiable away from the origin. The next result corresponds to a variational principle applied to integral functions.
				
					\begin{lemma}\label{lemmasubdeltaminimun}
					Let $z(\cdot)$ be a measurable function with values in $X$,  and let $\varepsilon(\cdot)$ and $\lambda(\cdot)$ be two strictly positive measurable functions. Suppose that  $z(t)$ is an $\varepsilon(t)$-minimum of $f_t$. Then there are measurable functions $y$ and $y^*$ such that for almost all $t \in T$, $y^*(t)\in \Asub f(t,y(t))$, $\| y(t) - z(t)\| \leq \lambda(t)$, $|f(t,y(t)) - f(t,z(t))| \leq \epsilon(t)$ and $\|y^*(t) \| \leq 4 \epsilon(t)/\lambda(t)$.
				\end{lemma}
				
				\begin{proof}
				Consider $\eta_i>0$ with $\eta_0=1$ such that $\sum_{i= 0}^\infty\eta_i=2$.  Then define $\eta_i(t):=\eta_i \cdot  \varepsilon(t)/\lambda^2(t)$, the space  $S=X\times\prod_{i=0}^{\infty}X$ with the product topology,  and the  function $\varphi :\T \times S \to \Rex$ given by $$\varphi(t,y,(x_i)) = \sum_{i=0}^{\infty}\eta_i(t)  \| y-x_i\|^2. $$ 
					It is not hard to prove that $S$ is a Polish space (i.e., metrizable, complete and separable) and that $\varphi$ is measurable.
					Consider the set valued map $A$ defined by $$y  \in A(t,(x_i)) \text{  if and only if }  y \in \textnormal{argmin}_{X} \bigg\{ f(t,\cdot)+ \varphi(\cdot, (x_i))\bigg\};$$ hence, by \cite[Theorem 8.2.11]{MR2458436} (see also \cite{MR1491362,MR1485775,MR0467310}), $A$
					 is a measurable set-valued map from $T\times \prod_{i=0}^{\infty}X$ to $X$, while by \cite[Proposition III.13]{MR0467310} we have that $\grafo A \in \mathcal{A}\otimes \mathcal{B} ( \prod_{i=0}^{\infty}X) \otimes \mathcal{B} (X)$. Furthermore, by \cite[Lemma 6.4.2]{MR2267655} (see also \cite[Proposition 1.49]{MR1485775}) we have that $\mathcal{B} ( \prod_{i=0}^{\infty}X) \otimes \mathcal{B} (X)=\mathcal{B}(S) $, which implies that $\grafo A \in \mathcal{A} \otimes \mathcal{B}(S)$.
					 
					 Now we consider   the multifunction $M$ so that $(y, (x_i)) \in M(t)$ if and only if  
						\begin{enumerate}[label=(\roman*),ref=(\roman*)]
						\item \label{M1}$\|(z(t) - y\| \leq \lambda(t)$, $\|x_i-y\|\leq \lambda(t)/\sqrt{2}^{i}$ for all $i=1,2,...$,
						\item\label{M2} $f(y)  + \varphi(t,y,(x_i))\leq  f(z(t))$,
						\item\label{M3} $f(t,w) +  \varphi(w,(x_i))  \geq f(t,y) +  \varphi(y,(x_i))$ for all $w \in X$.
						
					\end{enumerate}
					
					From the fact that  every function involved in the definition of  $M$ is $\mathcal{A} \otimes \mathcal{B}(S)$-measurable and the fact that \cref{M3} is equivalent to $(t,y,(x_i)) \in \grafo A$, we have $\grafo M \in \mathcal{A}\otimes \mathcal{B}(S)$. 
					
					  We claim that $M(t)$ is non-empty for all $t\in T$. Indeed, consider the type-gauge function $\rho(a,b):=\| a-b\|^2$. Then, applying the \emph{Borwein-Preiss Variational Principle} (see \cref{BorweinPreiss:VP}) to the $\epsilon(t)$-minimum $z(t)$ of $f$, with $\rho$ and the sequence $\eta_i(t)$, there exists $(y,(x_i))$ that verifies  \cref{M2,M3}. Furthermore, $\rho(z(t),y) \leq \epsilon(t)/\eta_0(t)= \lambda^2(t)$ and  $\rho(x_i,y) \leq \epsilon(t)/(2^i \eta_0(t))=  \lambda^2(t)/2^i$, which clearly implies \cref{M1}.

	Now, by the \emph{Measurable Selection Theorem} (see \cref{measurableselectiontheorem}), there exist measurable functions $(y(t), x_i(t)) \in M(t)$ ae. Hence, $\| y(t) - z(t)\| \leq \lambda(t)$,  $\| x_i(t) -y(t)\| \leq \lambda(t)/\sqrt{2}^{i}$ for all $i=1,2,...$,  $|f(t,y(t)) - f(t,z(t))| \leq \varepsilon(t)$  and $f(t,w) +  \varphi(t,w,(x_i(t)))  \geq f(t,y(t)) +  \varphi(t,y(t),(x_i(t)))$ for all $w \in X$ ae. 
					
					Finally, it is easy to see that $\phi(t,y) := \sum_{i=0}^\infty \eta_i(t) \| y - x_i(t)\|^2$  is $C^1$  with  respect to the second argument (see, e.g.,  \cref{lemmadifferenciabilidadnorma}), $\nabla \phi(t,y(t))$ is measurable and 
					\begin{align*}
						\| \nabla \phi(t,y(t)) \|&\leq \sum\limits_{i=0}^{\infty}2 \eta_i(t) \|y(t) - x_i(t)\| 
					\leq  4\frac{\epsilon(t)}{ \lambda(t)}.
						\end{align*}
					
				 Hence,  
					$f(t,\cdot)+ \phi(t,\cdot)$ attains a minimum at $y(t)$, and so $y^*(t):=-\nabla \phi(t,y(t))$ belongs to $\Asub f_t(y(t))$.
				\end{proof}

				The following result gives some preliminary consequences of the existence of a $p$-robust minimizer.
			\begin{lemma}\label{Robustlocalminimizer}
				Let $(\T,\mathcal{A},\mu)$ be a finite measure space, $p\in [1,+\infty)$ and  $x \in X$ be a $p$-robust  minimizer of $\Intf{f}$ on $B\subseteq X$. Then for every sequence $\varepsilon_n \to 0$,  and $\varepsilon_n$-minimizer $(x_n(\cdot),y_n)$ of  the function $\varphi_n: \textnormal{L}^p({\T},X)\times X \to \Rex$, defined as 
				\begin{align*}
					\varphi_n(w,u):=\displaystyle\int\limits_T  f(t,w(t)) d\mu(t) + n \int\limits_T \| w(t)-u\|^p d\mu(t) + \|  x_0 -u\|^p + \delta_B(u),
				\end{align*}
				 we have
				\begin{enumerate}[label={(\alph*)},ref={(\alph*)}]
					\item\label{aRobustlocalminimizer} $n (\| x_n(\cdot)-y_n \|_p)^p$, $ \| x_n(\cdot)-x \|_p$,  $\| y_n-x \|$ $\to 0$, and
					\item\label{bRobustlocalminimizer} $\displaystyle\int\limits_{{\T}} | f(t,x_n(t))  - f(t,x)|d\mu(t) \to 0$.
				\end{enumerate}
				In particular, we have  
				\begin{equation}\label{igualdadproblemasperturbados}
					\sup\limits_{n\in \N}\inf\limits_{\substack{w \in \textnormal{L}^p({\T},X)\\ u\in  X}} \varphi_n(w,u)= \Intf{f}(x).
				\end{equation}
			\end{lemma}
			\begin{proof}
				First, for $n \geq 1$ and $\gamma >0$  define 
				\begin{align*}
					\nu_n:=& \inf\{  \varphi_n(w,u) \mid  w \in \textnormal{L}^p({\T},X) \text{ and } u\in  X  \},\\
					\kappa_\gamma:=&\inf\{ \int_{\T} f(t,w(t)) \mid \int_{\T} \| w(t) - u\|^p d\mu(t) \leq \delta,\text{ }w \in \textnormal{L}^p({\T},X) \text{ and } u\in  B  \}.
				\end{align*}
				We have 
				$$n (\| x_n(\cdot)-y_n\|_p)^p \leq \int_{\T} \left( f(t,x_n(t)) +n\| x_n(t)-y_n  \|^p \right) d\mu(t)  + \| y_n -x\|^p \leq \Intf{f}(x) +\varepsilon_n.$$
				 The last inequality implies that $\int_{\T} \| x_n(t)-y_n\|^p d\mu(t) \to 0$, and so, setting $\gamma_n:=\int_{\T} \|x_n(t) -y_n\|^p d\mu(t) $,
				\begin{align*}
					\kappa_{\gamma_n}-\varepsilon_n \leq \int_{\T} f(t,x_n(t))d\mu(t) -\varepsilon_n \leq  \varphi_n(x_n,y_n) -\varepsilon_n \leq \nu_n\leq \Intf{f}(x_0).
				\end{align*}
				By taking the limits we conclude that $\displaystyle\int\limits_{{\T}}  f(t,x_n(t)) d\mu(t)  \to  \int\limits_{{\T}}f(t,x_0)d\mu(t) $, and, consequently,  \cref{aRobustlocalminimizer} and   \cref{igualdadproblemasperturbados} follow. Finally, \cref{bRobustlocalminimizer} follows by using  \cref{lemma:convergenciaL1}.
			\end{proof}

		\section{Some technical results}
		In this part of the appendix we present the proofs of some technical results, which were used  in the main sections of the article.
		\subsection{Sufficient conditions for $p$-robust minima}\label{proof_of_Prop:SuficientCondition}
		
		Here we present the proof of \cref{Prop:SuficientCondition}, which ensures conditions to have a $p$-robust infimum.
		
		\begin{proof}[Proof of \cref{Prop:SuficientCondition}]
			\begin{enumerate}[label={(\alph*)},ref={(\alph*)}] 
				\item 	In the first case define $$\nu_n:=\inf\limits_{\substack{w \in \textnormal{L}^p({T},X)\\ u\in  B}}\{ \displaystyle\int\limits_T f(t,w(t)) \mid \displaystyle\int\limits_T \| w(t) - u\|^p \leq 1/n \},$$ take $\varepsilon_n \to 0^+$ and $(x_n,y_n) \in \textnormal{L}^p({T},X)\times B$ such that 
				\begin{equation}\label{fatoudemostracion}
				-\varepsilon_n + \int\limits_{{T}} f(t,x_n(t)) \leq \nu_n,
				\end{equation}
				and $\displaystyle\int\limits_T \| x_n(t) - y_n\|^p \leq 1/n$. We can suppose that for every $t_1 \in {T}$ and $t_2 \in A$, $\| x_n(t_1) -y_n\| \to 0$ and $f(t_2,\cdot)$ is sequentially $\tau$-inf-compact. So, by Fatou's  lemma we have that, for every subsequence $x_{n_k}$ of $x_n$,
				\begin{align}\label{eq:fatou}
					\int_{T} \liminf f(t,x_{n_k}(t)) d\mu(t)  \leq \liminf \int_{T} f(t,x_{n_k}(t)) \leq  \inf_B \Intf{f} < +\infty.
				\end{align}
				Then,  in particular, for some $ t_0 \in A$, $\liminf f(t_0,x_n(t_0)) < + \infty$, and there exist a subsequence $x_{n_{k(t_0)}}(t_0)$ and a constant $M_{t_0}$ such that $f(t_0, x_{n_{k(t_0)}}(t_0))  \leq M_{t_0}$. Whence, by the inf-compactness  of $f(t_0,\cdot)$, there exists a subsequence $z_n$ of $x_{n_{k(t_0)}}(t_0)$ such that $z_n \to w_0 \in X$. Because $\| x_n(t_0) -y_n\| \to 0$, we get the existence of  a subsequence $y_{\phi(n)}$ of $y_n$ such that $y_{\phi(n)} \overset{\tau}{\to}  w_0 \in B$ (because $B$ is $\tau$-closed). Then, from the fact that $\| x_n(t) - y_n\| \to 0$, we get $x_{\phi(n)}(t) \overset{\tau}{\to} w_0$ for all $t\in {T}$. Finally, taking into account \cref{fatoudemostracion} and  using the lsc of the integrand in  \cref{eq:fatou} we obtain
				\begin{align*}
					\inf_B \Intf{f} &\leq \Intf{f}(w_0)  \leq \displaystyle\int\limits_T \liminf f(t, x_{\phi(n)}(t)) d \mu(t) \\ &\leq \liminf \displaystyle\int\limits_T  f(t, x_{\phi(n)}(t)) d \mu(t) \leq \wedge_{p,B} \Intf{f} \leq  \inf_B \Intf{f}.
				\end{align*}
				\item The second case follows from the first part, by modifying the measure space and the integrand function as follows: For $\omega_0 \notin T$, define $(\tilde{{T}},\tilde{\mathcal{A}},\tilde{\mu})$, where $\tilde{{T}}={T}\cup\{ \omega_0\}$, $\tilde{\mathcal{A}}=\sigma(\mathcal{A}, \{ \omega_0\})$ (the $\sigma$-algebra generated by $\mathcal{A}\cup \{  \{ \omega_0\}\}$) and $\tilde{\mu}(A)=\mu(A\backslash\{ \omega_0\}) + \1_{A}(\omega_0)$, and 
				\begin{align*}
					\tilde{f}(t,x)=\Bigg\{  \begin{array}{cl}
						f(t,x), &\text{ if } t\in T,\\
						\delta_{B} (x),&\text{ if } t=\omega_0.
					\end{array}
				\end{align*}
				Then, by the first part $	\wedge_{p,B} \Intf{\tilde{f}}^{\tilde{\mu}}= \inf_B \Intf{\tilde{f}}^{\tilde{\mu}}$ and, consequently,
				\begin{align*}
					\inf_B \Intf{f} \geq 	\wedge_{p,B} \Intf{f} \geq 	\wedge_{p,B} \Intf{\tilde{f}}^{\tilde{\mu}}= \inf_B \Intf{\tilde{f}}^{\tilde{\mu}}=\inf_B \Intf{f}.
				\end{align*}
				
			\item 	In the last case, let $K$ be the $q$-integrable Lipschitz constant and  consider $w\in \textnormal{L}^p({T},X)$ and $y \in B$. Then 
				\begin{align*}
					\displaystyle\int\limits_T f(t,w(t)) d\mu(t) &\geq - \int\limits_T| f(t,w(t))  - f(t,y)|d\mu(t) + \int\limits_T f(t,y) d\mu(t)  \\
					&\geq -\int\limits_T K(t)\| w(t)- y\| d\mu(t) +  \inf_B \Intf{f}.
				\end{align*}
				So, the result follows taking the appropriate  limits.
			\end{enumerate}
		
		\end{proof}

		\subsection{Proof of \cref{aproximatecalculsrules}}\label{Apendix:aproximatecalculsrules}
		We prove  \cref{aproximatecalculsrules} given in   \cref{section:Robusted}.
		\begin{proof}[Proof of \cref{aproximatecalculsrules}]
			We recall that the norm in $X$ is assumed to be $C^1$ away from the origin. Consider the function $\ell(x)=\| x\|^p$. It is easy to see that $\ell$ is $C^1$ everywhere. Moreover, we have that   $$ \| x \|^{p-1} \leq  \| \nabla \ell(x)\| \leq p \| x\|^{p-1}  \text{ for all } x\in X.$$ 
Consider $r  \in (0,1)$ such that $x_0$ is a $p$-robust minimizer of $\Intf{f}$ on $B:=\mathbb{B}(x_0,r)$, and fix  a family of $\eta_i >0$ such that $\eta_0=r$ and $\sum_{i=0}^{+\infty}\eta_i=1$. Now  define   $\varphi_n: \textnormal{L}^p({T},X)\times X \to \Rex$  by
			$$ \varphi_n(x,y)= \displaystyle\int\limits_T f(t,x(t)) d \mu(t) + n \displaystyle\int\limits_T  \ell(  x(t) -y ) d \mu(t) + \ell(y-x_0) +\delta_B(y).$$
			Then  \cref{Robustlocalminimizer} says that  
			$$\sup\limits_{n\in \N}\inf\limits_{\substack{w \in \textnormal{L}^p({T},X)\\ u\in  X}} \varphi_n(w,u)= \Intf{f}(x_0),$$
			and so there exists $\varepsilon_n \to 0^+$ (with $\varepsilon_n \in (0,\eta_0^2)$ for large enough $n$) such that $(x_0,x_0)$ is an $\varepsilon_n$-minimum of $\varphi_n$. Then, by applying the Borwein-Preiss Variational Principle (see \cref{BorweinPreiss:VP}) with the type-gauge function $\rho : (\textnormal{L}^p({T},X)\times X )^2  \to \R$ given by  $$\rho((w_1,u_1),(w_2,u_2)):=\int_{T} \ell (w_1(t) -w_2(t)) d\mu(t) + \ell(u_1 -u_2),$$ and the sequence $(\eta_i)_{i\geq 0}$, to the function $\varphi_n$ and the $\epsilon_n$-minimum $(x_0,x_0)$, we can find points $(x^n_i,y^n_i)_{i\in \N} , (x^n_\infty,y^n_\infty)  \in  \textnormal{L}^p({T},X)\times X$ such that:
			\begin{enumerate}[label={(\textnormal{BP.}\arabic*)},ref={(\textnormal{BP.}\arabic*)}]
				\item\label{Borwein-Preiss1}
				$\displaystyle\int\limits_T\ell (x_0 -x^n_\infty(t)) d\mu(t) + \ell(x_0 -y^n_\infty)\leq \frac{\varepsilon_n}{\eta_0}$, 
				$\displaystyle\int\limits_T\ell (x_i^n(t) -x^n_\infty(t)) d\mu(t) + \ell(y_i^n -y^n_\infty)\leq \frac{\varepsilon_n}{2^i\eta_0}$,
				\item\label{Borwein-Preiss2} $\varphi_n(x^n_\infty,y^n_\infty) + \phi_n(x^n_\infty,y^n_\infty)\leq \varphi_n(x_0,x_0)$, and
				\item\label{Borwein-Preiss3} $\varphi_n(w,u) + \phi_n(w,u)>\varphi_n(x^n_\infty,y^n_\infty) + \phi_n(x^n_\infty,y^n_\infty)$  for all $(w,u) \in  \textnormal{L}^p({T},X)\times X \backslash \{ (x^n_\infty,y^n_\infty)\}$,
			\end{enumerate}
			where 
			\begin{align*}
				\phi_n(w,u)= &\displaystyle\sum_{i=1}^{\infty} \eta_i \left( \int\limits_T   \ell( w(t)- x^n_i(t))  d \mu(t)\right)   + \sum_{i=1}^{\infty} \eta_i \ell (u- y^n_i )\\
				=& \displaystyle\int\limits_T \left( \sum_{i=1}^{\infty} \eta_i \ell( w(t)- x^n_i(t))  \right) d \mu(t) + \sum_{i=1}^{\infty} \eta_i \ell (u- y^n_i ).
			\end{align*}
			On the one hand, by \cref{Borwein-Preiss2} the quantity $\int_{T} h_n(t,x^n_\infty(t))d\mu(t)$ is finite, where $$h_n(t,v):= f(t,v)  + n \ell(v -  y^n_\infty) + \sum\limits_{i=1}^{\infty} \eta_i \ell (v- x^n_i(t))$$  is a normal  integrand functional, and by  \cref{Borwein-Preiss3} (taking $u=y^n_\infty$)
			\begin{align*}
				\displaystyle\int\limits_T h_n(t,x^n_\infty(t))d\mu(t)=& \inf\limits_{w \in \textnormal{L}^p({T},X)}  \int\limits_T  h_n(t,w(t)) d\mu(t)\\
			 =&  \int\limits_T  \inf\limits_{u\in X}h_n(t,u) d\mu(t),
			\end{align*}
			where the last equality is given  by \cref{dualitycastaing}. Then, by the sum rule (see, e.g., \cite[Exercise 3.1.12]{MR2144010}) we get
			\begin{equation}\label{equationmira}
			0 \in \Asub f(t,x_\infty^n(t))   + nu^*_n(t) + v_n^*(t) \; ae,
			\end{equation}
			where  $u^*_n(t):= \nabla \ell (x_\infty^n(t)- y_\infty^n) $ and $v^*(t):= \sum\limits_{i=1}^{\infty}  \eta_i  \nabla \ell (x_\infty^n(t)- x^n_i(t))$. The measurability  and differentiability  of these functions follow from   \cref{lemmadifferenciabilidadnorma} (notice that this infinite sum can also be seen as an integral functional). The  estimate of the gradient of the function $\ell$ gives us $\|u_n^*(t)\|^q\leq  p^q \| x_\infty^n(t)- y_\infty^n\|^p$  and $ \int\limits_T \| v_n^*(t)\|^q d \mu(t)  \to 0$. On the other hand,  by \cref{Borwein-Preiss3} (taking $w=x^n_\infty$),
			\begin{align*}
				n\int\limits_T   \ell( x_\infty^n(t) -  y^n_\infty) d\mu(t) +  \ell( y_\infty^n - x_0) + \sum\limits_{i=1}^{\infty} \eta_i \ell( y_\infty^n- y^n_i )\\
				= \inf\limits_{u \in X} \left( n\int\limits_T   \ell( x_\infty^n(t) -  u) d\mu(t) +  \ell( u - x_0) + \sum\limits_{i=1}^{\infty} \eta_i \ell( y_\infty^n- y^n_i )\right).
			\end{align*}
			Hence, again   \cref{lemmadifferenciabilidadnorma} gives us the differentiability of these three functions, and  after some  calculus yields  $ 0 = -n  \int_{T} u_n^*(t) d\mu(t) + w_n^* $ with $\| w_n^*\|\to 0$. Thus,  there exists $x_n^*:= -nu^*_n(t) - v_n^*(t)\in  \textnormal{L}^q({T},X^*)$ such that $x_n^*(t)\in\Asub f(t,x^n_\infty(t))$  (see Equation \cref{equationmira}),  and the previous computations give us  $$(\int_{T} \|x_n^*(t)\|^q)^{1/q} \leq n (\int_{T} \|u_n^*(t)\|^q)^{1/q}  + (\int_{T} \|v_n^*(t)\|^q)^{1/q} ,$$ and $\|  \int_{T} x_n^*(t) d\mu(t) \| \leq  \|  \int_{T} v_n^*(t) d\mu(t) \| + \|w^*_n\|  \to 0$.  By \cref{Borwein-Preiss2} we have  that $(x_\infty^n,y_\infty^n)$ is an $\varepsilon_n$-minimizer of $\varphi_n$, and so by   \cref{Robustlocalminimizer} we conclude that $n (\| x_\infty^n(t) - y_\infty^n\|_p)^p \to 0$, $  \int_{T} | f(t,x^n_\infty(t)) - f(t,x_0)|  d\mu(t)\to 0$. Finally,
			\begin{align*}
				\| x_n^*\|_q  \| x^n_\infty-y^n_\infty\|_p&\leq  \bigg( n (\displaystyle\int_{\T} \|u_n^*(t)\|^q d\mu(t))^{1/q}  + (\displaystyle\int_{\T} \|v_n^*(t)\|^qd\mu(t) )^{1/q} \bigg)  \| x^n_\infty-y^n_\infty\|_p \\ 
				&\leq  n p (\| x^n_\infty(\cdot)-y^n_\infty\|_p )^{p/q}  (\| x_n(\cdot)-y_n\|_p)  + \|v_n^*(\cdot)\|_q \| x^n_\infty(\cdot)-y^n_\infty\|_p\\
				&= n p  (\| x_n(\cdot)-y_n\|_p)^p  + \|v_n^*(\cdot)\|_q \cdot \| x_n(\cdot)-y_n\|_p \to 0.
			\end{align*} 
		\end{proof}

					\subsection{Proof of \textcolor{black}{\cref{separablereduction}}}\label{proofofseparablereduction}
						In this section, we give a proof of \cref{separablereduction}, by using some results on \emph{separable reductions} for the Fr\'echet subdifferential. 
						
						Let us  recall the concept of a \emph{rich family} in a non-separable Banach space. The symbol $S(X\times X^*)$ denotes the family of
						sets $U\times Y$ where $U$ and $Y$ are (norm-) separable closed linear
						subspaces of $X$ and $X^*$. A set $\mathcal{R} \subseteq S(X\times X^*)$ is
						called a \emph{rich family} if for every $U\times Y\in S(X\times X^*)$,
						there exists $V\times Z \in \mathcal{R}$ such that $U \subseteq V$, $Y
						\subseteq Z$ and $%
						\overline{ \bigcup_{n\in \mathbb{N} } U_n}\times \overline{ \bigcup_{n\in 
								\mathbb{N} } Y_n} \in \mathcal{R}$ whenever the sequence $(U_n\times
						Y_n)_{n\in \mathbb{N}} \subseteq \mathcal{R}$ satisfy $U_n \subseteq U_{n+1}$
						and $Y_n \subseteq Y_{n+1}$ (for more details, see \cite%
						{MR3447714,MR3582299,MR3533170} and references therein). In \cite[Theorem
						3.1]{MR3447714} the authors showed that there exists a rich family in non-separable
						Asplund spaces as in the following proposition.
						
						\begin{proposition}\cite[Theorem
							3.1]{MR3447714}
							\label{PROPSEPARABLE}  Let $(X, \|\cdot\| )$ be an Asplund space and let $f:X\to \mathbb{R} \cup \{ +\infty\}$ be any proper
							function. Then there exists a rich family $\mathcal{R} \subseteq S(X\times
							X^*)$ such that $Y_1 \subseteq Y_2$ whenever $V_1\times Y_1,V_2\times Y_2
							\in \mathcal{R}$ and $V_1 \subseteq V_2$, such that for
							every $V\times Y \in \mathcal{R}$, the assignment $Y \ni x^* \to x^*_{|_V}
							\in V^*$ is a surjective isometry from $Y$ to $V^*$, and for every $v \in V$
							we have that  
							\begin{align}  \label{SEPARABLEREDUC1}
								(\Asub f(v) \cap Y)_{|_V}=(\Asub f(v))_{|_V}=\Asub f_{|V}(v);
							\end{align}
							that is, in more details, if $v^*\in \Asub f_{|_V}(v)$, then there exists a
							unique $x^*\in \Asub f(v) \cap Y$ such that $x^*_{|_V}=v^*$ and $\| x^*\|
							=\| v^*\|$.
						\end{proposition}
						
						Besides, it has been proved that \emph{intersections of countably many rich
							families of a given space is (not only non-empty but even) rich} (see \cite[%
						Proposition 1.1]{MR1622793} or \cite[Proposition 1.2]{MR3533170}). Then for
						the case $(T,\mathcal{A})=(\mathbb{N},\mathcal{P}(\mathbb{N}))$ there must
						exist a rich family $\mathcal{R}$ for the integrand function $(f_n)$,
						satisfying the properties of  \cref{PROPSEPARABLE} and with \cref%
						{SEPARABLEREDUC1} uniformly for every $n\in \mathbb{N}$, as well as for the
						integral functional $\Intf{f}$. Using this family, we can extend all the previous
						statements to arbitrary Asplund spaces in the case when $(T,\mathcal{A})=(\mathbb{N}%
						,\mathcal{P}(\mathbb{N}))$.
						
					\begin{proof}[Proof of \cref{separablereduction}]
						Assume that the assumptions on $f,g$ in    \cref{theoremsubdiferential} hold in an Asplund space $X$ (the assumptions in  \cref{theoremsubdiferentialinfinity}, respectively). Let $x^*\in \Asub \Intf{f}(x)$ and $\rho$ be a $w^*$-continuous seminorm on $X^*$; for instance, $\rho= \max_{i=1,\ldots ,p} \langle \cdot, e_i\rangle $ with some $e_i \in X$. Then consider $V\times Y \in \mathcal{R}$ such that $x, e_i \in V$, $i=1,...,p$ and $x^*\in Y$. Then  $x^*_{|V}=:y^*\in \Asub (\Intf{f}){|_V}(x)$.   Then applying   \cref{theoremsubdiferential}, there  exist sequences  $y_n \in V$, $x_n \in \textnormal{L}^p(T,V)$, $z_n^* \in \textnormal{L}^q(T,V^*)$ such that:
							\begin{enumerate}[label={(\alph*)},ref={(\alph*)}]
								\item $z_n^*(t) \in  \Asub  f_{|_V}(t,x_n(t))$ ae,
								\item $\| x- y_n\| \to 0$, $\displaystyle \int_T \| x - x_n(t)\|^p d\mu(t) \to 0 $,
								\item $\| z_n^*(\cdot)\|_q \| x_n(\cdot)- y_n\|_p \to 0$, ($ \displaystyle\int_T \| x_n^*(t)\| \| x_n(t)- y_n\| d\mu(t) \to 0$ resp.)
								\item $\displaystyle\int_T \langle z_n^*(t), x_n(t)- x\rangle d\mu(t)\to 0$,  $\rho\left( \displaystyle \int_T z_n^*(t) d\mu(t) - x^*\right) \to 0$,
								\item $ \displaystyle \int_T| f(t,x_n(t)) - f(t,x)| d\mu(t)\to 0$.
							\end{enumerate}
							Then define $x_n^*(t)$ as the unique element in $\Asub f(t,x_n(t))\cap Y$ such that $\| x_n^*(t)\| =\| z_n^*(t)\|$ and $(x_n^*(t))_{|_V}=z^*_n(t)$.  Now  $\| y_n^*(\cdot)\|_q=\| x_n^*(\cdot)\|_q$, which implies that $$\| x_n^*(\cdot)\|_q \| x_n(\cdot)- y_n\|_p \to 0$$ ($ \displaystyle\int_T \| x_n^*(t)\| \| x_n(t)- y_n\| d\mu(t) \to 0$, respectively) and $x_n^* \in \textnormal{L}^q(T,X^*)$.  From the fact that $x_n(t),y_n, e_i, x \in V$  and $(x_n^*(t))_{|_V}=z^*_n(t)$ we conclude that 
						$\int_T \langle x_n^*(t), x_n(t)- x\rangle d\mu(t)\to 0$ and $\rho\left(  \int_T x_n^*(t) d\mu(t) - x^*\right) \to 0$. Then, the sequences $y_n$, $x_n(\cdot)$ and $x_n^*(\cdot)$ satisfy the required properties.  
					\end{proof}
				\subsection{Fatou-type lemma}\label{Section_Fatou_type_MordukhovichSub}
				The next lemma gives the first sequential approximation of subgradients in the limiting subdifferential of the $\Intf{f}$. This results uses the sequential representation of the Fr\'echet subdifferential found in \cref{theoremsubdiferentialinfinity}, together with a Fatou's lemma for sequences in Banach spaces. This will allow to get a relation between the limits of the integral of subgradients in the Fr\'echet subdifferential and the integral of the limiting subdifferential. 
				
				\begin{lemma}
					\label{MordukhovichSub}  Consider $x^* \in \sub \Intf{f}(x)$, $y^* \in \sub^\infty \Intf{f}(x)$, a finite family of linearly independent
					points $\{ e_i\}_{i=1}^p$, $W:=\spn \{ e_i\}$ and $\rho(\cdot) :=\max\{ |
					\langle \cdot , e_i \rangle | \}$. Then there exist sequences $x_n,y_n \in 
					\textnormal{L}^\infty({T},X)$, $x_n^*,y_n^* \in \textnormal{L}^1({T}%
					,X^*)$ and $\lambda_n \to 0^+$ such that:  
					\begin{enumerate}[label={(\alph*)},ref={(\alph*)}]
						\item \label{MordukhovichSuba} $x_n^*(t) \in \Asub f(t,x_n(t))$ ae,
						\item \label{MordukhovichSubb}$\ \| x - x_n(\cdot)\|_{\infty} \to 0 $,
						\item \label{MordukhovichSubc} $\rho(\displaystyle \int\limits_T x_n^*(t) d\mu(t) - x^*)\to 0$,
						\item \label{MordukhovichSubd}$ \lim \int\limits_T | f(t,x_n(t)) - f(t,x)|  d \mu(t)= 0$,
					\end{enumerate}
					as well as
					\begin{enumerate}[label={(\alph*$^\infty$)},ref={(\alph*$^\infty$)}]
						\item \label{MordukhovichSubasing} $y_n^*(t) \in \Asub f(t,y_n(t))$ ae,
						\item \label{MordukhovichSubbsing}$ \| x - y_n(\cdot)\|_{\infty} \to 0 $,
						\item \label{MordukhovichSubcsing}$\rho (\lambda_ n \cdot \displaystyle \int\limits_T y_n^*(t) d\mu(t) - y^*)\to 0$,
						\item \label{MordukhovichSubdsing}$ \lim \int\limits_T | f(t,y_n(t)) - f(t,x)|  d \mu(t)= 0$,
						\item $y^* \in (B(y^*_n) \cap W)^{-}$, where $$B(y^*_n):= \{ u \in X :  \liminf \int\limits_T \langle y_n^* ,u \rangle^{+}d\mu  <+\infty \}.$$  
					\end{enumerate}
					
					Moreover, if there exists a bounded sequence $x_n^*$ (in $\textnormal{L}^1({T},X^*)$) (or $\lambda_n y_n^*$, respectively) satisfying the above
					properties, then  
					\begin{align*}
					x^* \in \displaystyle\int\limits_{{T}} \sub f(t,x) d\mu(t) +
					C((x^*_n))^- + W^{\perp}, \\
					 y^* \in \displaystyle\int\limits_{{T}} \sub^\infty f(t,x) d\mu(t) + C(\lambda_n y^*_n)^- +
					W^{\perp} \text{ (respectively) },
					\end{align*}
					where $C(x^*_n):=\{ u \in X : ( \langle x_n^*(\cdot),u
					\rangle^{+})_{n\in \mathbb{N}} \text{ is uniformly integrable}\}$. 
				\end{lemma}
				
				\begin{proof}
					By the definition of  $\sub \Intf{f}(x)$ and $\sub^\infty \Intf{f}(x)$, there exist  sequences $z^*_n \in \Asub \Intf{f}(z_n)$, $s_n^* \in  \Asub \Intf{f}(s_n)$ and $\lambda_n \to 0^+$   such that $z_n, s_n \overset{\Intf{f}}{\to} x$,   $z_n^* \overset{w^*}{\to} x^*$ and $\lambda_n \cdot s_n^* \overset{w^*}{\to} y^*$. Whence, by   \cref{theoremsubdiferentialinfinity} (and using a diagonal argument) there exist  sequences $x_n,y_n \in \textnormal{L}^\infty({T},X)$, $x_n^*,y_n^* \in \textnormal{L}^1({T},X^*)$  such that
					\begin{enumerate}[label={(\roman*)},ref={(\roman*)}] 
						\item $x_n^*(t) \in \Asub f(t,x_n(t))$ ae,
						\item $ \| x - x_n(\cdot)\|_\infty \to 0$,
						\item \label{MordukhovichSub:3}$\rho( \int_{\T} x_n^*(t) d\mu(t) - x^*) \to 0$,
						\item \label{MordukhovichSub:4}$ \int_{\T}   f(t,x_n(t))d\mu(t) \to   \int_{\T}  f(t,x) d\mu(t)$,
					\end{enumerate}
					
					\begin{enumerate}[label={(\roman*$^\infty$)},ref={(\roman*$^\infty$)}]
						\item  $y_n^*(t) \in \Asub f(t,y_n(t))$ ae,
						\item $\| x - y_n(\cdot)\|_\infty  \to 0$,
						\item \label{MordukhovichSub:sing3}$\rho( \lambda_n \cdot \int_{\T} y_n^*(t) d\mu(t) - y^*) \to 0$,
						\item \label{MordukhovichSub:sing4}$\int_{\T}  f(t,y_n(t))d\mu(t) \to \int_{\T}  f(t,x)  d\mu(t)$.
					\end{enumerate} 
					So, by  \cref{lemma:convergenciaL1} we conclude that $$ \lim \int_{T} | f(t,x_n(t)) - f(t,x)|  d \mu(t)= 0$$ and $$ \lim \int_{T} | f(t,y_n(t)) - f(t,x)|  d \mu(t)= 0.$$ 
					Moreover, if we take  $u \in B(y^*_n)\cap W$, then 					$$\langle y^* , u \rangle = \lim \lambda_n \int \langle y_n ^*,u \rangle d\mu(t)\leq  \liminf \lambda_n \int \langle y_n(t) ^*,u \rangle^+ d\mu(t) =0.$$ 
					This proves the first part. To prove  the second part, consider  a continuous projection  $P_W :X \to W$,  and  assume that $\sup\limits_{n} \int\limits_{{T}} \| x_n^*(t) \| d\mu(t) < \infty$. So, by \cite[Corollary 4.1]{MR2197293} we have that 
					\begin{align*}
						\Ls^{w^*}\{ \displaystyle  \int\limits_{{T}} x_n^*(t) d\mu(t)    \} \subseteq  \displaystyle  \int\limits_{{T}} \Ls^{w^*}\{ x_n^*(t)\} d\mu(t) +  C(x^*_n)^{-} + W^{\perp},
					\end{align*} 
					where $ \Ls^{w^*}\{ x_n^*(t)\}$ represents the sequential upper limit of the sequence $(x_n^*(t))$. Moreover, if $\sup\limits_{n} \int\limits_{{T}} \| x_n^*(t) \| d\mu(t)$ is finite, then (up to  subsequences)  $$w_n ^*:= \int_{T} x^*_n(t) d\mu(t)\to w^*_0.$$ Moreover, $w_n^* = P^*_W(w^*_n) + w^*_n - P^*_W(w^*_n)$ and by \cref{MordukhovichSub:3} we get that $P^*_W(w^*_n)  \to P^*_W(x^*)$. Therefore, $w_0^* - P^*_W(x^*)\in W^{\perp}$, and then we conclude that $x^* =  w_0^* + P^*_W(x^*)-w_0^* + x^* - P^*_W(x^*) \in  \int\limits_{{T}} \Ls^{w^*}\{ x_n^*(t)\} d\mu(t) +  C(x^*_n)^{-} + W^{\perp}$.

					 Finally, we have to prove that $\Ls^{w^*}\{ x_n^*(t)\}  \subseteq  \sub f(t,x)$ for ae $t$. Indeed, from the previous convergence (and taking a subsequence if necessary, see, e.g., \cite[Theorem 13.6]{MR1321140}) we can take a measurable set $\tilde{{T}} $ such that $\mu({T}\backslash \tilde{{T}})=0$  and for every $t \in \tilde{{T}}$, $ x_n(t) \to x_0$ and $f(t,x_n(t) ) \to f(t,x)$.  Then take an integrable selection $a^*(t) \in \Ls^{w^*}\{ x_n^*(t)\}$ and fix $t_0 \in \tilde{{T}}$. So, there exists a subsequence $x^*_{n_{k(t_0)}}(t_0) \to a^*(t_0) $, and then $x_{n_{k(t_0)}}(t_0) \to x_0$ and $f(t_0,x_{n_{k(t_0)}}) \to f(t,x)$. Hence,  $a^*(t_0)\in \sub f(t_0,x)$. The case for the point $y^*$ is similar, and so we omit the proof.
				\end{proof}
			
				\subsection{Characterization of  Integrable compact sole property}\label{characterizationofcompactsole}
				The final lemma allows us to understand the definition of integrable compact soles in terms of the primal space instead of the dual space, using
				the polar cone.
				
				\begin{lemma}
					\label{compactsoloiffnonemptyinterior}  Let $C:{T}\rightrightarrows X$ be a
					measurable multifunction with non-empty $w$-closed convex values, let $e \in
					X$ and let $\delta >0$ . Then the following statement are equivalent: 
					
					\begin{enumerate}[label={(\alph*)},ref={(\alph*)}]
						\item\label{lemma:parta} For every measurable selection $c^*$  of   $C^-(t)(:=(C(t))^-)$ one has $$\delta^{-1} \langle c^*(t) , -e\rangle \geq \|c^*(t) \| \;ae.$$
						\item\label{lemma:partb} The ball in $L^{\infty}({\T},X)$ around $e$  with radius  $\delta$ is contained in $$\{x(\cdot) \in \textnormal{L}^\infty({\T},X) : x(t)\in C(t) \text{ ae}\}.$$
					\end{enumerate}
				\end{lemma}
				
				\begin{proof}
					First, by Castaing's representation  there exist  sequences of measurable selections $c_n$ and $c_n^*$ of  $C$ and $C^-$ respectively such that $C(t)=\overline{\{c_n(t)\}}^{\norma}$ and $C^{-}(t)=\overline{\{c^*_n(t)\}}^{w^*}$  (the measurability of   $C^-$ follows from the fact that  $C^-(t )= \bigcap_{n\in\N} \{ x^* \in X^* : \langle c_n(t) ,x^* \rangle \leq 0   \}$).

					 Suppose that \cref{lemma:parta} holds and consider $h \in \textnormal{L}^\infty({T},X)$ and  $\| h\|_\infty\leq 1$. Then $\langle e+ \delta h(t),c_n^*(t) \rangle \leq  \langle e ,c_n^*(t) \rangle +  \delta \| c_n^*(t)\|\leq 0$ ae. Since the last inequality holds for all $n$ and the measurable selection $c^*_n(t)$ is dense in $C^-(t)$,  we conclude that $e+\delta h(t)\in (C^{-}(t))^-$ ae.  Then, using the Bipolar theorem  (see, e.g., \cite[Theorem 3.38]{MR2766381}), we obtain that $e+\delta h(t)\in C(t)$ ae. 
					 
					 Now suppose that \cref{lemma:partb} holds and consider a measurable selection $c^*$ of $C^-$ and $(0,1)\ni \varepsilon_n \to 0$. Take a measurable selection $h_n(t) \in B(0,1)$ such that $\langle h_n(t),c^*(t)\rangle \geq \| c^*(t)\| -\varepsilon_n$ ae. So, $e + \delta h(t) \in C(t)$ ae, and hence $\langle c^*(t), e + \delta h(t) \rangle \leq 0$ ae. Therefore, $\delta^{-1} \langle c^*(t) , -e\rangle \geq \|c^*(t) \|- \varepsilon_n$ ae. From the stability of  null sets under countable intersections we get $\delta^{-1} \langle c^*(t) , -e\rangle \geq \|c^*(t) \|$ for ae $t$.
				\end{proof} 
				
\bibliographystyle{siamplain}
\bibliography{references}
\end{document}